\documentclass[12pt,reqno]{amsart}
\usepackage{amscd,amssymb,amsmath,multicol,enumerate}
\usepackage[all]{xy}

\newtheorem{thm}[equation]{Theorem}
\numberwithin{equation}{section}
\newtheorem{cor}[equation]{Corollary}

\newtheorem{rmk}[equation]{Remark}

\newtheorem{lem}[equation]{Lemma}

\newtheorem{defin}[equation]{Definition}
\newtheorem{diag}[equation]{Diagram}
\newtheorem{prop}[equation]{Proposition}

\begin{document}
\raggedbottom \voffset=-.7truein \hoffset=0truein \vsize=8truein
\hsize=6truein \textheight=8truein \textwidth=6truein
\baselineskip=18truept
\def\vareps{\varepsilon}
\def\mapright#1{\ \smash{\mathop{\longrightarrow}\limits^{#1}}\ }
\def\mapleft#1{\ \smash{\mathop{\longleftarrow}\limits^{#1}}\ }
\def\mapup#1{\Big\uparrow\rlap{$\vcenter {\hbox {$#1$}}$}}
\def\mapdown#1{\Big\downarrow\rlap{$\vcenter {\hbox {$\ssize{#1}$}}$}}
\def\on{\operatorname}
\def\ot{\otimes}
\def\a{\alpha}
\def\bz{{\Bbb Z}}
\def\gd{\on{gd}}
\def\imm{\on{imm}}
\def\sq{\on{Sq}}
\def\g{\gamma}
\def\eps{\epsilon}
\def\tfrac{\textstyle\frac}
\def\w{\wedge}
\def\b{\beta}
\def\A{{\cal A}}
\def\P{{\cal P}}
\def\zt{{\Bbb Z}_2}
\def\bq{{\Bbb Q}}
\def\ker{\on{ker}}
\def\coker{\on{coker}}
\def\im{\on{im}}
\def\wt{\on{wt}}
\def\u{{\cal U}}
\def\e{{\cal E}}
\def\ext{\on{Ext}}
\def\wbar{{\overline w}}
\def\xbar{{\overline x}}
\def\ybar{{\overline y}}
\def\zbar{{\overline z}}
\def\ebar{{\overline e}}
\def\nbar{{\overline n}}
\def\mbar{{\overline m}}
\def\ubar{{\overline u}}
\def\ft{{\widetilde f}}
\def\pt{{\widetilde P}}
\def\Ct{L}
\def\rt{{\widetilde R}}
\def\vt{{\widetilde\nu}}
\def\bot{{\widetilde{bo_2}}}
\def\ni{\noindent}
\def\coef{\on{coef}}
\def\den{\on{den}}
\def\N{{\Bbb N}}
\def\Z{{\Bbb Z}}
\def\Q{{\Bbb Q}}
\def\R{{\Bbb R}}
\def\C{{\Bbb C}}
\def\bc{{\bold C}}
\def\bd{{\bold D}}
\def\Ah{\widehat{A}}
\def\Bh{\widehat{B}}
\def\Ch{\widehat{C}}
\def\tmf{\on{tmf}}
\def\TMF{\on{TMF}}
\def\tmfbar{\overline{\tmf}}
\def\extt{\widetilde{\ext}}
\title[Connective versions of $TMF(3)$]
{Connective versions of $TMF(3)$}

\author{Donald M. Davis}
\address{Department of Mathematics, Lehigh University\\Bethlehem, PA 18015, USA}
\email{dmd1@lehigh.edu}
\author{Mark Mahowald}
\address{Department of Mathematics, Northwestern University\\Evanston, IL 60208, USA}
\email{markmah@mac.com}
\date{May 20, 2010}


\maketitle
\begin{abstract} We study three connective versions of the spectrum for topological modular forms of level 3.
All three were described briefly by Mahowald and Rezk in \cite{MR}, but we add much detail to their discussion.
Letting $\tmf(3)$ denote our connective model which is a ring spectrum, we compute $\tmf(3)_*(RP^\infty)$.
 \end{abstract}
\section{Introduction}\label{intro}
In \cite{MR}, the second author and Rezk discuss the periodic spectrum $\TMF(\Gamma_0(3))$, abbreviated here as $\TMF(3)$,
associated to topological modular forms of level 3. In Section 7 of \cite{MR}, they discuss briefly three
connective models of $\TMF(3)$. The main purpose of this paper is to clarify and fill in details for these connective models.

The first model is $X\w\tmf$, where $X$ is a certain 10-cell complex; it was first introduced by the second author and Gorbounov in their study of $MO[8]$.
It is probably the best of our three models because it is a ring spectrum. In Section \ref{GMsec},
we define it and compute its homotopy groups. Our description and method of computation differ somewhat
from that of \cite{GM}.

In \cite{MR}, another connective model for $\TMF(3)$ is discussed, which is $Z\w\tmf$, where $Z$ is a certain
8-cell complex. Although $Z\w\tmf$ is not a ring spectrum, its importance is primarily that the dimensions of the cells of $Z$ allow one to easily construct a
map $Z\to \TMF(3)$ thanks to certain homotopy groups of $\TMF(3)$ being 0. The other models are then related to
$\TMF(3)$ via the $Z$-model. In Section \ref{sec2}, we provide some additional details to the sketch given in \cite{MR}.

In Section \ref{BGsec}, we consider a third model which was also introduced in \cite{MR}. This one is closely
related to consideration of a splitting of $\tmf\w\tmf$. There is a Brown-Gitler-type splitting of the $A$-module $H^*(\tmf\w\tmf)$, and we show that
it is not realized by a spectrum splitting. Again we add some clarity and detail to the description in \cite{MR} of this model and its
homotopy groups.

All three of our models are equivalent after inversion of $v_2$, but as connective models they are different.
The homotopy groups of the second and third models are subsets of those of the first, obtained by omitting certain
initial portions. One nice feature of our approach is to relate the Ext calculations for the second and third models
directly to that of the first, even though the constructions of the spectra are very different.

In Section \ref{Psec} we compute $\pi_*(P_1\w X\w\tmf)$, where $P_1=RP^\infty$. If we think of $X\w\tmf$ as our best
model of $\tmf(3)$, then this is $\tmf(3)_*(P_1)$. Our original goal in undertaking this study was to use $\TMF(3)$ in
obstruction theory, and this computation would be a first step toward doing that.

\section{The model of Gorbounov and Mahowald}\label{GMsec}
In their study of $\pi_*(MO[8])$ in \cite{GM}, the second author and Gorbounov introduced a new spectrum,
which turns out to be the best model for a connective version of $\TMF(3)$. Certain aspects of the construction
in \cite{GM} were unclear to the first author, and so we have prepared this alternative account.
In Theorem \ref{GMthm} we define the spectrum, and in Theorems \ref{piXtmf} and \ref{v1v2thm} we determine its homotopy
groups. In Section \ref{sec2}, we will establish its relationship with $\TMF(3)$.

\begin{thm}\label{GMthm} \begin{enumerate}[(a)]
\item There is a 9-cell CW complex $Y$ with one cell of each dimension $0$, $2$, $3$, $4$, $6$, $7$, $8$, $9$, and $10$, in which the following
Steenrod operations are nonzero on the bottom class $g$:
\begin{equation}\label{Y}\sq^2,\,\sq^3,\,\sq^4,\,\sq^4\sq^2,\,\sq^5\sq^2,\,\sq^6\sq^2=\sq^8,\,\sq^6\sq^3,\,\sq^7\sq^3.\end{equation}
This together with $\sq^6g=0$ completely describes $H^*(Y)$ as an $A$-module.
\item There is a map $\Sigma^3Y\mapright{\a}S^0$ extending $2\nu$.
\item Let $X$ denote the mapping cone of $\a$. There is a map $X\mapright{f}bo$ which is the identity on the bottom cell.
\item Let $\ft$ denote the composite
$$X\w\tmf\mapright{f\w1}bo\w\tmf\mapright{\mu}bo,$$
and let $C$ denote the mapping cone of $\ft$. There is an isomorphism of $A$-modules
$$H^*(C)\approx\Sigma^4A/A(\sq^4,\sq^5\sq^1).$$
\item $X\w\tmf$ is a ring spectrum.
\end{enumerate}
\end{thm}
\begin{rmk}{\rm This $X\w\tmf$ will be  our preferred model for the connective $\tmf(3)$, because it is a ring spectrum. The spectrum $\Sigma^{16}X\w \tmf$ is apparently a
subspace of $MO\langle8\rangle/\tmf $,
but this will not enter into our argument. This was the motivation for the initial
discussion of $X\w\tmf$ in \cite{GM}.}\end{rmk}

Throughout the paper, $A_n$  denotes the subalgebra of the mod 2 Steenrod algebra $A$ generated by $\sq^i$ for $i\le 2^n$.
Also $\eta$ and $\nu$ denote the (class of the) Hopf maps in the 1- and 3-stems. All cohomology groups have coefficients in $\zt=\bz/2$.
Our spectra are localized at 2.

\begin{proof} (a.) Let $X_3=S^0\cup_\eta e^2\cup_2e^3$ and $X_7=S^0\cup_\nu e^4\cup_\eta e^6\cup_2 e^7$. Let $Q$ denote the quotient
of $X_3\w X_7$ by its 4-skeleton. The Steenrod algebra structure, or equivalently the cell structure, of $Q$ is depicted in Diagram
\ref{Zdiag}. Here a symbol $(i,j)$ is the product class or cell of an $i$-cell of $X_3$ and a $j$-cell of $X_7$. We indicate both $\sq^1$ and $\sq^2$
by straight lines, and $\sq^4$ by a curved line.
\bigskip

\begin{minipage}{6.5in}
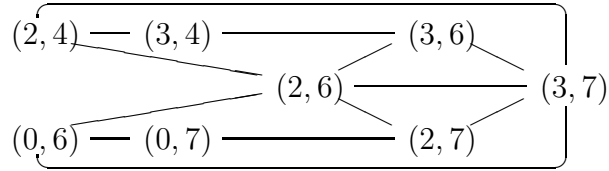
\begin{diag}\label{Zdiag}{Cell structure of quotient of $X_3\w X_7$}
\begin{center}
\begin{picture}(430,80)
\def\mp{\multiput}
\put(80,20){$(0,6)$}
\put(130,20){$(0,7)$}
\put(230,20){$(2,7)$}
\put(280,40){$(3,7)$}
\put(80,60){$(2,4)$}
\put(130,60){$(3,4)$}
\put(180,40){$(2,6)$}
\put(230,60){$(3,6)$}
\mp(110,24)(0,40){2}{\line(1,0){15}}
\mp(160,24)(0,40){2}{\line(1,0){65}}
\put(204,50){\line(2,1){20}}
\put(204,39){\line(2,-1){20}}
\put(254,60){\line(2,-1){20}}
\put(254,29){\line(2,1){20}}
\put(103,60){\line(6,-1){72}}
\put(210,44){\line(1,0){65}}
\put(103,29){\line(6,1){72}}
\put(190,70){\oval(200,10)[t]}
\put(290,70){\line(0,-1){18}}
\put(190,18){\oval(200,10)[b]}
\put(290,18){\line(0,1){18}}
\end{picture}
\end{center}
\end{diag}
\end{minipage}
\medskip

There is a map $g$ from this $Q$ to $S^6\cup_2 e^7\cup_\eta e^9$ which sends the cells $(2,4)$, $(3,4)$, and $(3,6)$ by
degree 1,
and  the cells $(0,6)$, $(0,7)$, and $(2,7)$ by degree $-1$. The  fiber of the composite
$$X_3\w X_7\mapright{\text{coll}}Q\mapright{g} S^6\cup_2 e^7\cup_\eta e^9$$
is the desired complex $Y$. The Steenrod operations in $Y$ can be determined from the Cartan formula together with the
fact that $\sq^2$ and $\sq^3$ are nonzero in $X_3$, and $\sq^4$, $\sq^6$, and $\sq^7$ are nonzero in $X_7$. For example,
$\sq^4\sq^2$ on the bottom class is $(2,4)$, while $\sq^6$ is $(2,4)+(0,6)$, which is $g^*(x_6)$ and hence is 0 in the fiber.

(b.) Let $DY$ denote the $S$-dual of $Y$, with cells of dimensions the negative of those of $Y$. Thus the top cell of $DY$ has dimension 0.
Note that $\sq^8=0$ in $H^*(DY)$, since it is dual  to $\chi\sq^8$, which is 0 in $H^*(Y)$.
Let $(DY)^{(-1)}$ denote the $(-1)$-skeleton of $DY$. We use the ASS to show that
$2\nu$ is in the image of $\pi_{3}(DY)\mapright{c_*}\pi_3(S^0)$, where $c$ collapses $(DY)^{(-1)}$. We use Bruner's software (\cite{Br})
to compute $\ext_A^{s,t}(H^*(DY))$ for $2\le t-s\le 4$ as in Diagram \ref{extdiag}. Here and throughout, we omit
writing $\zt$ as the second argument of our Ext groups.

\medskip

\begin{minipage}{6.5in}
\begin{diag}\label{extdiag}{Ext groups for $2\le t-s\le 4$}
\begin{center}
\begin{picture}(430,180)(20,0)
\def\mp{\multiput}
\def\elt{\circle*{2.5}}
\put(35,150){$\ext_A(H^*((DY)^{(-1)}))$}
\mp(149,151)(106,0){2}{$\longrightarrow$}
\put(375,151){$\mapright{\delta}$}
\put(176,150){$\ext_A(H^*DY)$}
\put(296,150){$\ext_A(H^*S^0)$}
\mp(50,10)(120,0){3}{\line(1,0){60}}
\mp(80,30)(120,0){2}{\line(0,1){100}}
\mp(80,30)(0,20){6}{\elt}
\mp(200,30)(0,20){6}{\elt}
\mp(80,30)(120,0){2}{\line(1,1){20}}
\mp(100,50)(120,0){2}{\elt}
\mp(103,30)(0,20){3}{\elt}
\put(203,50){\line(0,1){40}}
\mp(203,50)(0,20){3}{\elt}
\mp(223,30)(0,20){3}{\elt}
\put(203,50){\line(1,1){20}}
\put(83,90){\elt}
\put(60,50){\elt}
\put(46,45){$B'$}
\put(320,30){\line(0,1){40}}
\mp(320,30)(0,20){3}{\elt}
\put(309,26){$B$}
\put(309,46){$A$}
\put(300,50){\elt}
\put(300,50){\line(1,1){20}}
\mp(56,0)(120,0){3}{$2$}
\mp(76,0)(120,0){3}{$3$}
\mp(96,0)(120,0){3}{$4$}
\put(203,50){\circle{5}}
\end{picture}
\end{center}
\end{diag}
\end{minipage}
\medskip

The desired class $2\nu$ is indicated with $A$ in the diagram, and is the image of the circled
class. The class $\nu$, indicated by $B$, maps to $B'$ in the exact sequence.

(c.) Let $DX$ denote the $S$-dual of $X$, with 10 cells, in dimensions $-14$ up to 0. Then $[\Sigma^iX,bo]\approx
\pi_i(DX\w bo)$, and this can be computed by the ASS with $E_2=\ext_{A_1}(H^*DX)$. The $A_1$-structure of $H^*(DX)$
is easily seen, and the $\ext_{A_1}$-calculation easily made, giving the result in Diagram \ref{extA1diag}
in dimension $<4$. There are clearly no possible differentials, and our desired map is detected in filtration 0
by the circled element.

\medskip
\begin{minipage}{6.5in}
\begin{diag}\label{extA1diag}{$\ext_{A_1}(H^*DX)$ in $t-s<4$}
\begin{center}
\begin{picture}(430,160)
\def\mp{\multiput}
\def\elt{\circle*{2.5}}
\put(40,15){\line(1,0){300}}
\put(60,15){\elt}
\put(50,0){$-14$}
\put(90,30){\vector(0,1){120}}
\mp(150,15)(120,0){2}{\vector(0,1){135}}
\mp(150,15)(120,0){2}{\line(1,1){30}}
\put(142,0){$-8$}
\put(165,15){\line(1,1){15}}
\put(153,45){\vector(0,1){105}}
\put(153,45){\line(1,1){28}}
\put(208,45){\vector(0,1){105}}
\put(211,60){\vector(0,1){90}}
\put(214,90){\vector(0,1){60}}
\put(273,60){\vector(0,1){90}}
\put(273,60){\line(1,1){30}}
\put(276,75){\vector(0,1){75}}
\put(276,75){\line(1,1){28}}
\put(279,105){\vector(0,1){45}}
\put(279,105){\line(1,1){26}}
\put(268,0){$0$}
\put(270,15){\circle{5}}
\end{picture}
\end{center}
\end{diag}
\end{minipage}
\medskip

(d.) There is a commutative diagram in which horizontal and vertical sequences are fiber sequences.
$$\begin{CD} S^0\w\tmf @>=>> \tmf @. @.\\
@VVV @VVV @. @.\\
X\w\tmf @>\ft>> bo @>>> C @.\\
@VVV @VVV @V=VV @.\\
(X/S^0)\w\tmf@>\widehat f>> bo/\tmf @>>> C@>>> \Sigma(X/S^0)\w\tmf
\end{CD}$$

 The restriction of $\ft$ to the 4-skeleton is $S^0\cup_{2\nu}e^4\to S^0\cup_{\nu}e^4$ of degree 1 on the
 bottom cell. Thus $\widehat f$ has degree 2 on its bottom 4-cell.
 The $A$-module $H^*(bo/\tmf)$ is isomorphic to $A\otimes_{A_2}\overline{A_2/\!/A_1}$, and the $A_2$-module
$\overline{A_2/\!/A_1}$ has basis
\begin{equation}\label{C}\{g_4,\,\sq^2g_4,\,\sq^3g_4,\,\sq^4\sq^2g_4,\,\sq^4\sq^3g_4,\,\sq^6\sq^3g_4,\,\sq^4\sq^6\sq^3g_4\}.
\end{equation}
Thus $(\widehat f)^*=0$, and, since $X/S^0=\Sigma^4Y$, there is a short exact sequence of $A$-modules
$$0\to H^*(\Sigma^5Y\w\tmf)\to H^*(C)\to A\otimes_{A_2}\overline{A_2/\!/A_1}\to0,$$
and $\sq^1g_4\ne0$ in $H^*C$. We conclude that $H^*(C)$ is an extended cyclic $A_2$-module on a 4-dimensional generator,
with nonzero operations being those in (\ref{C}) and $\sq^1$ and the operations listed in (\ref{Y}) applied to
$\sq^1$. One easily checks that this $A_2$-module equals $A_2/(\sq^4,\sq^5\sq^1)$, and so the $A$-module $H^*(C)$
is as claimed.

(e.) We will prove there is a map $m':X\w X\to bo$ extending the inclusion of the bottom cell and that when followed by the map
$bo\to C$ of part (d), the composite is trivial. Thus by the definition of $C$, $m'$ factors through a map
$m:X\w X\to X\w\tmf$ extending the inclusion of the bottom cell. Smashing this twice with $\tmf$ and following by two multiplications
of $\tmf$ yield the desired product on $X\w\tmf$.

We construct the dual of $m'$, an element of $\pi_0(DX\w DX\w bo)$. The $E_2$-term of the ASS converging to $\pi_*(DX\w DX\w bo)$
is $\ext_{A_1}(H^*(DX\w DX))$. The $A_1$-structure of $H^*(DX)$ is easily seen and $\ext_{A_1}$ of tensor products of the summands
is easily computed, as, for example, in \cite{thesis}. We obtain that in the vicinity of $t-s=0$, the chart has a copy of $bo_*$
beginning in position (0,0) and 15 additional copies of $bo_*$ beginning in positions $(0,s)$ for $3\le s\le 12$. The groups in
$t-s=-1$, i.e. corresponding to $\pi_{-1}$, are all 0. Thus there are no possible differentials from $t-s=0$ in the ASS, and we deduce the existence
of our map $S^0\to DX\w DX\w bo$, whose dual is $m'$.

Next we compute the ASS for $\pi_*(DX\w DX\w C)$. Let $Y$ be as in part (a). Then $DX=\Sigma^{-4}DY\cup_{2\nu}e^0$, and
so $H^*(DX)\approx H^*(\Sigma^{-4}DY)\oplus H^*(S^0)$ as $A$-modules. Thus the ASS converging to $\pi_*(DX\w DX\w C)$ has
\begin{eqnarray*}E_2&\approx& \ext_A(H^*(\Sigma^{-4}DY\w DY)\otimes A/(\sq^4,\sq^{5,1}))\oplus \ext_A(H^*(DY)\otimes A/(\sq^4,\sq^{5,1}))\\
&&\oplus \ext_A(H^*(DY)\otimes A/(\sq^4,\sq^{5,1}))\oplus \ext_A(\Sigma^4 A/(\sq^4,\sq^{5,1})).\end{eqnarray*}
Note that the bottom class of $DY$ is in grading $-10$.
We can use Bruner's software to see that each of these Ext groups is 0 in $t-s=0$. For example,
$$\ext_A(H^*(\Sigma^{-4}DY\w DY)\otimes A/(\sq^4,\sq^{5,1}))$$ has 15 $\bz$-towers in the $(-3)$-stem, beginning in filtrations 2 through 8.
It is 0 in stems $-2$, $-1$, and 0, and then in the 1-stem has 21 $\bz$-towers, on each of which $\eta$ and $\eta^2$ are nonzero.

Thus $\pi_0(DX\w DX\w C)=0$ and hence $[X\w X,C]=0$. Therefore the map $X\w X\mapright{m'}bo\to C$ is trivial, implying the result
by the argument of the first paragraph of the proof.
\end{proof}

The main step toward describing $\pi_*(X\w\tmf)$ is, because of \ref{GMthm}(d), the Ext calculation in Theorem \ref{extCthm}.
This calculation was first made in \cite{GM}, but our approach will be somewhat different. Our approach will be useful in
performing other related Ext calculations. The description is in terms of $bo_*$ and $bsp_*$, which are depicted in Diagram \ref{bodiag}.

\bigskip

\begin{minipage}{6.5in}
\begin{diag}\label{bodiag}{$bo_*$ and $bsp_*$}
\begin{center}
\begin{picture}(400,135)(30,0)
\def\mp{\multiput}
\def\elt{\circle*{1.75}}
\mp(-5,10)(200,0){2}{\line(1,0){140}}
\mp(-2,0)(200,0){2}{$0$}
\mp(38,0)(200,0){2}{$4$}
\mp(78,0)(200,0){2}{$8$}
\mp(116,0)(200,0){2}{$12$}
\mp(0,10)(200,0){2}{\vector(0,1){100}}
\mp(0,10)(80,40){2}{\line(1,1){20}}
\mp(0,10)(10,10){3}{\elt}
\mp(80,50)(10,10){3}{\elt}
\put(40,40){\vector(0,1){70}}
\put(80,50){\vector(0,1){60}}
\put(120,80){\vector(0,1){30}}
\put(240,20){\vector(0,1){90}}
\mp(240,20)(80,40){2}{\line(1,1){20}}
\mp(240,20)(10,10){3}{\elt}
\mp(320,60)(10,10){3}{\elt}
\put(280,50){\vector(0,1){60}}
\put(320,60){\vector(0,1){50}}
\put(40,120){$bo_*$}
\put(240,120){$bsp_*$}
\mp(130,90)(5,3){3}{$\cdot$}
\mp(360,90)(5,3){3}{$\cdot$}
\end{picture}
\end{center}
\end{diag}
\end{minipage}
\medskip

We will denote by  $a_{x,y}$ an element of $\ext^{y,x+y}$. This corresponds to the usual $(x,y)$ components in an ASS.
There are standard elements $h_1$, $h_2$, and $v_2^4$ of $(x,y)$-grading $(1,1)$, $(3,1)$, and $(24,4)$, respectively.
Here and throughout, $R[a]\langle b_1,\ldots,b_r\rangle$ denotes a free module over a polynomial algebra $R[a]$ with basis
$\{b_1,\ldots,b_r\}$.
\begin{thm}\label{extCthm} As a bigraded abelian group, $\ext^{*,*}_A(A/A(\sq^4,\sq^5\sq^1),\zt)$ is isomorphic to
\begin{eqnarray*}&&\zt[v_2^8]\langle a_{0,0},\, h_2a_{0,0},\, a_{14,2},\,h_1a_{14,2},\,h_2a_{14,2},\,a_{31,5},\,
h_2a_{31,5},\,a_{39,7}\rangle\\
&\oplus&\ker(bo_*[v_2^4]\langle a_{5,1},a_{21,3}\rangle\to \zt[v_2^8]\langle a_{21,3}\rangle)\\
&\oplus &bsp_*[v_2^4]\langle a_{9,2},a_{17,4}\rangle.\end{eqnarray*}
\end{thm}
\begin{proof} By the Change-of-Rings Theorem, it is equivalent to compute $$\ext_{A_2}(A_2/A_2(\sq^4,\sq^5\sq^1),\zt).$$
One can verify that there is an exact sequence of $A_2$-modules:
\begin{eqnarray*} 0&\leftarrow& A_2/(\sq^4,\sq^{5,1})\mapleft{d_0} A_2\mapleft{d_1} \Sigma^4A_2\oplus\Sigma^6A_2/\!/A_1\mapleft{d_2}
\Sigma^{11}A_2/(\sq^1,\sq^5)\oplus\Sigma^{16}A_2\\
&\mapleft{d_3}&\Sigma^{18}A_2/(\sq^3)\oplus\Sigma^{20}A_2\mapleft{d_4}
(\Sigma^{25}A_2\oplus\Sigma^{26}A_2)/(\sq^1I_{25},\sq^3I_{25}+\sq^2I_{26})\\
&\mapleft{d_5}&\Sigma^{34}A_2/\!/A_1\oplus\Sigma^{36}A_2/(\sq^3)
\mapleft{d_6}\Sigma^{40}A_2\\
&\mapleft{d_7}&\Sigma^{46}A_2/(\sq^3)\oplus \Sigma^{52}A_2/\!/A_1\mapleft{d_8}\Sigma^{56}A_2/(\sq^4,\sq^{5,1})\leftarrow0
\end{eqnarray*}
with
\begin{eqnarray*}d_1(I_4)&=&\sq^4\\
d_1(I_6)&=&\sq^5\sq^1\\
d_2(I_{11})&=&\sq^7I_4\\
d_2(I_{16})&=&(\sq^{6,6}+\sq^{7,5})I_4+\sq^{4,6}I_6\\
d_3(I_{18})&=&\sq^2I_{16}+\sq^7I_{11}\\
d_3(I_{20})&=&\sq^4I_{16}+\sq^{6,3}I_{11}\\
d_4(I_{25})&=&\sq^7I_{18}+\sq^5I_{20}\\
d_4(I_{26})&=&\sq^{7,1}I_{18}+\sq^6I_{20}\\
d_5(I_{34})&=&\sq^{2,7}I_{25}\\
d_5(I_{36})&=&(\sq^{5,6}+\sq^{6,5})I_{25}+\sq^{4,6}I_{26}\\
d_6(I_{40})&=&\sq^4I_{36}+\sq^6I_{34}\\
d_7(I_{46})&=&\sq^6I_{40}\\
d_7(I_{52})&=&\sq^{7,5}I_{40}\\
d_8(I_{56})&=&\sq^4I_{52}+(\sq^{4,6}+\sq^{6,3,1})I_{46}.
\end{eqnarray*}

For $0\le i\le7$, let $C_i$ denote the $A_2$-module which is the domain of $d_i$. Because the domain of $d_8$ is $\Sigma^{56}$
of the beginning module, the exact sequence could be continued periodically with the $\Sigma^{56}A_2/(\sq^4,\sq^{5,1})$
removed, and $C_{i+8}\approx\Sigma^{56}C_i$. There is a spectral sequence building $\ext(A_2/(\sq^4,\sq^{5,1}))$
from $\bigoplus\limits_{i\ge0}\phi^i\ext(\Sigma^{-i}C_i)$, where $\phi^i$ increases filtration by $i$. Of the modules
that appear in $C_i$, $\ext(A_2)$ is just $\zt$ in $(0,0)$, $\ext(A_2/\!/A_1)$ is $bo_*$, $\ext(A_2/(\sq^1,\sq^5))$
is $bsp_*$, $\ext(A_2/(\sq^3))$ is $\ext(A_2)\oplus\phi\ext(\Sigma^2bsp_*)$, and $\ext((A_2\oplus\Sigma^1A_2)/(\sq^1I_0,
\sq^3I_0+\sq^2I_1))$ is $\phi^{-1}(\ker(bo_*\to\zt))$. When these are put together, one obtains exactly the claim
of the theorem. There can be no differentials because differentials are $h_i$-natural. The differentials would go
from position $(x,y)$ of $\phi^i\ext(\Sigma^{-i}C_i)$ to position $(x-1,y+1)$ of $\phi^{i+r}\ext(\Sigma^{-(i+r)}C_{i+r})$. In Diagram \ref{bigdiag},
we depict this chart for $x\le48$,  to show  the impossibility of differentials in both this SS converging to Ext,
and in an ASS to be considered later. Note that the $\zt$ in the 48-stem is $v_2^8$ times the initial $\zt$.
\end{proof}

\medskip
\begin{minipage}{6.5in}
\begin{diag}\label{bigdiag}{$\ext_A(A/(\sq^4,\sq^{5,1}))$ through degree $48$}
\begin{center}
\begin{picture}(400,510)(30,0)
\def\mp{\multiput}
\def\elt{\circle*{1.75}}
\put(-5,315){\line(1,0){400}}
\put(-2,302){$0$}
\put(-40,310){$s=0$}
\put(73,302){$5$}
\put(135,302){$9$}
\put(192,302){$13$}
\put(253,302){$17$}
\put(315,302){$21$}
\put(375,302){$25$}
\mp(0,315)(45,15){2}{\elt}
\put(0,315){\line(3,1){45}}
\put(75,330){\vector(0,1){180}}
\put(75,330){\line(1,1){30}}
\mp(75,330)(15,15){3}{\elt}
\put(134,345){\vector(0,1){165}}
\put(137,375){\vector(0,1){135}}
\put(194,360){\vector(0,1){150}}
\put(194,360){\line(1,1){31}}
\mp(194,360)(15.5,15.5){3}{\elt}
\put(197,390){\vector(0,1){120}}
\put(197,390){\line(1,1){28}}
\mp(197,390)(14,14){3}{\elt}
\mp(210,345)(15,15){2}{\elt}
\put(210,345){\line(1,1){15}}
\put(210,345){\line(3,1){45}}
\put(255,360){\elt}
\put(252,375){\vector(0,1){135}}
\put(255,405){\vector(0,1){105}}
\put(258,435){\vector(0,1){75}}
\put(309 ,375){\vector(0,1){135}}
\put(312,390){\vector(0,1){120}}
\put(312,390){\line(1,1){32}}
\mp(312,390)(16,16){3}{\elt}
\put(315,420){\vector(0,1){90}}
\put(315,420){\line(1,1){30}}
\mp(315,420)(15,15){3}{\elt}
\put(318,450){\vector(0,1){60}}
\put(318,450){\line(1,1){28}}
\mp(318,450)(14,14){3}{\elt}
\mp(330,375)(15,15){2}{\elt}
\put(330,375){\line(1,1){15}}
\put(372,405){\vector(0,1){105}}
\put(375,435){\vector(0,1){75}}
\put(378,465){\vector(0,1){45}}
\put(381,495){\vector(0,1){15}}
\put(-5,0){\line(1,0){400}}
\put(-6,5){$29$}
\put(75,5){$33$}
\put(154,5){$37$}
\put(234,5){$41$}
\put(313,5){$45$}
\put(-40,10){$s=5$}
\put(-5,15){\vector(0,1){270}}
\put(-5,15){\line(4,3){44}}
\mp(-5,15)(22,16.5){3}{\elt}
\put(-2,45){\vector(0,1){240}}
\put(-2,45){\line(4,3){42}}
\mp(-2,45)(21,15.75){3}{\elt}
\put(1,75){\vector(0,1){210}}
\put(1,75){\line(4,3){40}}
\mp(1,75)(20,15){3}{\elt}
\put(4,105){\vector(0,1){180}}
\put(4,105){\line(4,3){38}}
\mp(4,105)(19,14.25){3}{\elt}
\put(7,135){\vector(0,1){150}}
\put(7,135){\line(4,3){36}}
\mp(7,135)(18,13.5){3}{\elt}
\mp(40,15)(60,15){2}{\elt}
\put(40,15){\line(4,1){60}}
\put(77,30){\vector(0,1){255}}
\put(80,60){\vector(0,1){225}}
\put(83,90){\vector(0,1){195}}
\put(86,120){\vector(0,1){165}}
\put(89,150){\vector(0,1){135}}
\put(92,180){\vector(0,1){105}}
\put(154,45){\vector(0,1){240}}
\put(154,45){\line(4,3){44}}
\mp(154,45)(22,16.5){3}{\elt}
\put(157,75){\vector(0,1){210}}
\put(157,75){\line(4,3){42}}
\mp(157,75)(21,15.75){3}{\elt}
\put(160,105){\vector(0,1){180}}
\put(160,105){\line(4,3){40}}
\mp(160,105)(20,15){3}{\elt}
\put(163,135){\vector(0,1){150}}
\put(163,135){\line(4,3){38}}
\mp(163,135)(19,14.25){3}{\elt}
\put(166,165){\vector(0,1){120}}
\put(166,165){\line(4,3){36}}
\mp(166,165)(18,13.5){3}{\elt}
\put(169,195){\vector(0,1){90}}
\put(169,195){\line(4,3){34}}
\mp(169,195)(17,12.75){3}{\elt}
\put(200,45){\elt}
\put(231,60){\vector(0,1){225}}
\put(234,90){\vector(0,1){195}}
\put(237,120){\vector(0,1){165}}
\put(240,150){\vector(0,1){135}}
\put(243,180){\vector(0,1){105}}
\put(246,210){\vector(0,1){75}}
\put(249,240){\vector(0,1){45}}
\put(311,45){\vector(0,1){240}}
\put(311,45){\line(4,3){46}}
\mp(311,45)(23,17.25){3}{\elt}
\put(314,75){\vector(0,1){210}}
\put(314,75){\line(4,3){44}}
\mp(314,75)(22,16.5){3}{\elt}
\put(317,105){\vector(0,1){180}}
\put(317,105){\line(4,3){42}}
\mp(317,105)(21,15.75){3}{\elt}
\put(320,135){\vector(0,1){150}}
\put(320,135){\line(4,3){40}}
\mp(320,135)(20,15){3}{\elt}
\put(323,165){\vector(0,1){120}}
\put(323,165){\line(4,3){38}}
\mp(323,165)(19,14.25){3}{\elt}
\put(326,195){\vector(0,1){90}}
\put(326,195){\line(4,3){36}}
\mp(326,195)(18,13.5){3}{\elt}
\put(329,225){\vector(0,1){60}}
\put(329,225){\line(4,3){34}}
\mp(329,225)(17,12.75){3}{\elt}
\put(332,255){\vector(0,1){30}}
\put(332,255){\line(4,3){32}}
\mp(332,255)(16,12){3}{\elt}
\put(380,60){\elt}
\end{picture}
\end{center}
\end{diag}
\end{minipage}
\medskip

The following result is an easy consequence of Theorems \ref{GMthm} and \ref{extCthm}.
\begin{thm}\label{piXtmf} There is an isomorphism of graded abelian groups
\begin{eqnarray*}\pi_*(X\w \tmf)&\approx&bo_*[v_2^4]\langle  v_1v_2\rangle\oplus\ker(bo_*[v_2^4]\to\zt[v_2^8]\langle
v_2^4\rangle)\\
&&\oplus bsp_*[v_2^4]\langle 2v_2^2,2v_1v_2^3\rangle\\
&&\oplus\zt[v_2^8]\langle \nu,\,\nu^2,\,x,\,\eta x,\,\nu x,\,x^2,\,\eta x^2,\,y\rangle,
\end{eqnarray*}
where the (homotopy group, Adams filtration) of elements is $(2,1)$ for $v_1$, $(6,1)$ for $v_2$, $(17,3)$ for $x$,
and $(42,8)$ for $y$.\end{thm}
\begin{proof} We use the exact sequence $$\to\pi_*(\Sigma^{-1}C)\to \pi_*(X\w\tmf)\to\pi_*(bo)\to$$ from \ref{GMthm}.
The bottom class of $H^*(C)$ causes $\sq^4\iota_0=0$ in $H^*(X\w\tmf)$, and so {\it a chart for a spectral sequence
converging to $\pi_*(X\w\tmf)$
can be formed from $bo_*$ of Diagram \ref{bodiag} together with
 Diagram \ref{bigdiag} shifted by $(3,1)$ units}. By $h_i$-naturality there are no differentials or extensions, and so
 the chart depicts $\pi_*(X\w\tmf)$. The names $v_1v_2$, $2v_2^2$, and $2v_1v_2^3$ which we give to certain generators
 are, at least at this point, meant to only describe stem and filtration.
 \end{proof}

Our next result simplifies the $bo_*$-$bsp_*$-part of this description and also incorporates as much as we
can say about the ring structure from our approach.  Our limitation is that our approach can only give
the ring structure of $\pi_*(X\w\tmf)$ up to elements of higher filtration in the Adams-type spectral sequence
we have been using. Note that we say ``Adams-type" because we have elevated the filtrations of the part from $C$
by 1 compared to an actual ASS. The reason that we can't say any better than ``up to higher filtration" is,
first of all the usual limitation of an ASS, and secondly that our multiplication of $X\w\tmf$ is only
defined up to maps of higher filtration.  It seems that such deviations would change the product
structure in $\pi_*(X\w\tmf)$. For example, the product of classes that we call $2v_2^2$ and $2v_1^4v_2^2$
(so-called because of their image in $BP_*$; note that these classes are generators---the
elements without the factor 2 are not present in $\pi_*(X\w\tmf)$) would naturally be $4v_1^4v_2^4$, an element
which would be divisible by 4 in $\pi_*(X\w\tmf)$.  However, we cannot assert that this product of generators
is divisible by 4; it might equal, for example, $4v_1^4v_2^4+v_1^{16}$.

\begin{thm} \label{v1v2thm} There is an isomorphism of graded abelian groups
$$\pi_*(X\w\tmf)\approx K\oplus\zt[v_2^8]\langle \nu,\,\nu^2,\,x,\,\eta x,\,\nu x,\,x^2,\,\nu x^2,\,v_1v_2x^2\rangle,$$
where
$$K=\ker(R\to\zt[v_2^8](v_2^4))$$
with $R$ the subring of $\bz[v_1,v_2,\eta]/(2 \eta, \eta^3)$ generated by
$2v_1^2$,  $v_1^4$,  $v_1 v_2$,  $2v_2^2$,  and $v_2^4$. The isomorphism is, up to elements of higher filtration,
an isomorphism of rings, with the additional relations
$v_1^4x=\eta v_1^3v_2^3$, $v_1v_2x=\eta v_2^4$,  $x^3=\nu v_2^8$,
 and $x^7=0$.
\end{thm}

Stems of elements are as in Theorem \ref{piXtmf}.
Note that $x^7=0$, not just up to elements of higher filtration, as it lies in a zero group.

\begin{proof} It is not difficult to check that this description
is consistent as an Adams-filtered graded abelian group  with the description in Theorem \ref{piXtmf}.
We must establish various product formulas.

First we show that $x^2$ is nonzero, corresponding to $a_{31,5}$ in \ref{extCthm}. The multiplication of $X\w\tmf$ restricts to a filtration-1 map $\Sigma^{-1}C\w\Sigma^{-1}C\to\Sigma^{-1}C$.
Note that $H^*(\Sigma^{-1}C)\approx\Sigma^3A/(\sq^4,\sq^{5,1})$, and so the multiplication can be considered as
\begin{equation}\label{Cmult}\ext_{A_2}^{s,t}(A_2/(\sq^4,\sq^{5,1}))\otimes\ext_{A_2}^{s,t}(A_2/(\sq^4,\sq^{5,1}))\to
\ext_{A_2}^{s+1,t+4}(A_2/(\sq^4,\sq^{5,1}))\end{equation}
with $\iota_0\otimes\iota_0\mapsto h_2\iota_0$. With $x\in\ext^{2,16}_{A_2}(A_2/(\sq^4,\sq^{5,1}))$, the image of $x\otimes x$ is in
$\ext^{5,36}_{A_2}(A_2/(\sq^4,\sq^{5,1}))$. We wish to show it is nonzero.

Note that $x\iota_0\otimes\a\iota_0\mapsto \a h_2x\iota_0$, so we want the Yoneda product of $h_2x$ with $x$.
Using the minimal ``resolution" in the proof of \ref{extCthm}, we consider the following diagram:
$$\begin{CD}C_3@<<<C_4@<<<C_5\\
@V h_2x VV @V f_4 VV @V f_5 VV\\
\Sigma^{20}C_0 @<<< \Sigma^{20}C_1 @<<< \Sigma^{20}C_2\\
@. @. @V xVV\\
@. @. \Sigma^{36}\zt.\end{CD}$$
The relevant parts are
$$\begin{CD} \Sigma^{20}A_2 @<\sq^5,\sq^6<< (\Sigma^{25}A_2\oplus\Sigma^{26}A_2)/R@<\sq^{5,6}+\sq^{6,5},\sq^{4,6}<<
\Sigma^{36}A_2\\
@V1VV @V f_4 VV @V f_5 VV\\
\Sigma^{20}A_2@<\sq^4,\sq^{5,1}<< \Sigma^{24}A_2\oplus \Sigma^{26}A_2/\!/A_1 @<\sq^{6,6}+\sq^{7,5},\sq^{4,6}<<
\Sigma^{36}A_2\\
@. @. @VVV\\
@. @. \Sigma^{36}\zt.\end{CD}$$
We find that $f_4(\iota_{25})=\sq^1\iota_{24}$ and $f_4(\iota_{26})=\sq^2\iota_{24}+\iota_{26}$, and then that $f_5$ is the identity.

A similar argument works to show $x^3=\nu v_2^8$. Relations for $x^4$, $x^5$, and $x^6$ can be deduced from the stated relations.

 The elements
$\eta x$, $x^2$, and $y$ generate the three occurrences of $A_2/(\sq^3)$ in the resolution in the proof of \ref{extCthm}.
The $bsp_*$'s on $a_{17,4}$, $v_2^4a_{9,2}$, and $v_2^4a_{17,4}$ in Theorem \ref{piXtmf} are obtained from $\ext_{A_2}$
of these three $A_2/(\sq^3)$'s by omitting the initial $\zt$. This implies
that $v_1^4\eta x=\eta^2v_1^3v_2^3$, $v_1^4x^2=\eta^2v_1^2v_2^6$, and $v_1^4y=\eta^2v_1^3v_2^7$. One of our relations
is obtained by dividing the first of these by $\eta$, while the latter two imply that $y=v_1v_2x^2$.

The elements which we call $v_1^{4i}v_2^{8j}\cdot(2v_2^2)^e$ with $1\le e\le3$
in $E_2^{*,*}(H^*X)$ are in the image of the ring map from $\ext_{A_2}(\zt)$, and so products among them
are as we claim because of the ring structure of $\ext_{A_2}(\zt)$.
That the products of $(v_1v_2)^i$ with $2v_2^2$ are as claimed can be proved by a Yoneda product argument with the element
$2v_2^2$ of $\ext_{A_2}^{3,15}(\zt)$. To verify  this using a minimal resolution of $A/(\sq^4,\sq^{5,1})$, one should expand the
efficient resolution used in the proof of \ref{extCthm} to use only $A_2$ and $A_2/(\sq^1)$ (and not the more efficient
$A_2/\!/A_1$ and $A/(\sq^3)$). This produces some additional $\sq^4\sq^6$ terms in the resolution.
The following not-quite-commutative diagram of not-quite-exact sequences shows the most relevant terms in the morphism
from a portion of the resolution built on $(v_1v_2)^i$ to the most relevant terms of the resolution of $\zt$, and can be used to
establish that the Yoneda product of the element that we call $(v_1v_2)^i$ followed by the element that we call $2v_2^2$ equals the
element that we call $2v_1^iv_2^{i+2}$.
$$\begin{CD}A_2/(\sq^1)@<\sq^{4,6}<< \Sigma^{10}A_2@<\sq^2<< \Sigma^{12}A_2@<\sq^3<<\Sigma^{15}A_2\\
@V1VV @V\sq^{4,2}VV @V\sq^4VV @V1VV\\
A_2/(\sq^1)@<\sq^4<< \Sigma^4A_2@<\sq^4<< \Sigma^8A_2@<\sq^7<< \Sigma^{15}A_2.\end{CD}$$
For example, when $i=2$, the top row of this diagram corresponds to elements in $(13,3)$, $(22,4)$, $(23,5)$, and $(25,6)$ in Diagram \ref{bigdiag}.

To see that the elements that we call $v_1^iv_2^i$ multiply by one another as the notation suggests, we consider the morphism of minimal resolutions
inducing (\ref{Cmult}).
Let $$\bc:C_0\leftarrow C_1\leftarrow\cdots\qquad\text{(resp. }\bd:D_0\leftarrow D_1\leftarrow\cdots)$$ be a
minimal $A_2$-resolution of
$\Sigma^3A_2/(\sq^4,\sq^{5,1})$ (resp. $\Sigma^2\ker(A_2\to A_2/(\sq^4,\sq^{5,1})$).
Then (\ref{Cmult}) is induced by a morphism $\bd\mapright{\psi}\bc\otimes\bc$.
The class which we call $v_1^iv_2^i$ is dual to a generator $\a_i\in(C_{2i-1})_{10i-1}$ and to a generator
$\b_i\in(D_{2i-2})_{10i-2}$.

First we show that the square of our $v_1v_2$ class equals our class called $v_1^2v_2^2$.
The relevant parts are that $C_1\mapright{d_1} C_0$ has $C_0=\Sigma^3A_2$, $C_1=\Sigma^7A_2\oplus\Sigma^9A_2/\!/A_0$,
with $d_1(\iota_7)=\sq^4\iota_3$ and $d_1(\iota_9)=\sq^{5,1}\iota_3$, while the relevant part of  $\bd$  is
$$\Sigma^{18}A_2\mapright{d_2}\Sigma^{13}A_2/\!/A_0\mapright{d_1}\Sigma^6A_2$$
with $d_2(\iota_{18})=\sq^5\iota_{13}$ and $d_1(\iota_{13})=\sq^7\iota_6$. In the commutative diagram of exact sequences
$$\begin{CD} D_2@>d_2>> D_1@>d_1>> D_0\\
@V f_2 VV @V f_1 VV @V f_0 VV\\
C_1\otimes C_1 @>>> C_0\otimes C_1\oplus C_1\otimes C_0 @>>> C_0\otimes C_0
\end{CD}$$
we must have
\begin{eqnarray*}f_0(\iota_6)&=&\iota_3\otimes\iota_3\\
f_1(\iota_{13})&=&(1+T)(\iota_3\otimes\sq^3\iota_7+\sq^1\iota_3\otimes(\sq^2\iota_7+\iota_9)+\sq^2\iota_3\otimes\sq^1\iota_7
+\sq^3\iota_3\otimes\iota_7)\\
f_2(\iota_{18})&=&\iota_9\otimes\iota_9,\end{eqnarray*}
implying the result. Note that the important term here was the $\iota_9$, which occurred because of the difference between
$\sq^6$ and $\sq^2\sq^4$.

Now we show that the class which we call $v_1^2v_2^2$ times the class which we call $v_1^iv_2^i$ equals the class that we
call $v_1^{i+2}v_2^{i+2}$. This, with the result of the preceding paragraph, implies that all powers of $v_1v_2$
are as claimed.

The class which we call $2v_1^iv_2^{i+2}$ is dual to a generator $\g_{i+1}\in(C_{2i+2})_{10i+14}$ and to a
generator $\delta_{i+1}\in(D_{2i+1})_{10i+13}$.
In the resolutions, $d(\a_{i+1})\equiv\sq^5\g_i$ and  $d(\b_{i+1})\equiv\sq^5\delta_i$ mod other terms, where $\a_{i+1}$ and $\b_{i+1}$ are dual to $v_1^{i+1}v_2^{i+1}$ as above.
Because the product of our $2v_2^2$ class and our $v_1^iv_2^i$ class equals our $2v_1^iv_2^{i+2}$ class,
as was shown earlier, we conclude that in $\bd\mapright{\psi}\bc\otimes\bc$, $\psi(\delta_{i+1})=\g_1\otimes\a_i$ plus other
terms.  Thus, modulo other terms,
we have
$$d(\psi(\b_{i+2}))= \psi(d(\b_{i+2}))\equiv\sq^5\g_1\otimes\a_i$$
and $$d(\a_2\otimes\a_i)\equiv\sq^5\g_1\otimes\a_i,$$
from which we conclude $\psi(\b_{i+2})=\a_2\otimes\a_i$, which is equivalent to our claim.

Now that we know that the classes which we have named by monomials in $v_1$ and $v_2$ multiply consistently with
these names, we can deduce the final relation $v_1v_2x=\eta v_2^4$ from $v_1^4x=\eta v_1^3v_2^3$ by multiplying the latter
by $v_1v_2$ and then dividing by $v_1^4$.
\end{proof}

\section{An 8-cell model related to $\TMF(3)$}\label{sec2}
In \cite[\S7]{MR}, another connective model for $\TMF(3)$ is discussed, which is $Z\w\tmf$, where $Z$ is a certain
8-cell complex. Although $Z\w\tmf$ is not a ring spectrum, it is still true that $v_2^{-1}Z\w\tmf\simeq \TMF(3)$.
The importance of this model is primarily that the dimensions of the cells of $Z$ allow one to construct a
map $Z\to \TMF(3)$ thanks to certain homotopy groups of $\TMF(3)$ being 0. The other models are then related to
$\TMF(3)$ via the $Z$-model. In this section, we provide some additional details to the sketch given in \cite{MR}.

Let $X_7=S^0\cup_\nu e^4\cup_\eta e^6\cup_2 e^7$ be as in the proof of Theorem \ref{GMthm}.a,
and let $X_{421}=\Sigma^7DX_7=S^0\cup_2e^1\cup_\eta e^3\cup_\nu e^7$.
\begin{lem} \label{MRlem} The map $S^6\mapright{\nu^2} S^0\hookrightarrow X_{421}$ extends to a map
$\Sigma^6X_7\to X_{421}$.\end{lem}

The proof in \cite{MR}, using the ASS of $X_{421}$ through dimension 13, is perfectly clear.

\begin{defin} Let $Z$ denote the mapping cone of the map $\Sigma^{23}X_7\to \Sigma^{17}X_{421}$ obtained from Lemma \ref{MRlem}.
\end{defin}

\begin{prop}\label{toTMF3} There is an element $x\in\pi_{17}(\TMF(3))$ of order 2 which is not divisible by $\eta$,
and a map $Z\to \TMF(3)$ which extends this map $x$.\end{prop}
\begin{proof} This is where we need input from the theory of topological modular forms.
In \cite{MR}, a 48-periodic ring spectrum $\TMF(3)$ (called there $\TMF(\Gamma_0(3))$)
is defined and its homotopy groups calculated, using a spectral sequence defined using results
about elliptic curves. Their result (\cite[4.1]{MR}), localized at 2, is a $v_2^8$-inverted
version of our Theorem \ref{v1v2thm}, but with their ring structure being precise, not just up to
elements of higher filtration. We emphasize that our Theorem \ref{v1v2thm} and \cite[4.1]{MR} are
totally independent calculations. Our \ref{v1v2thm} uses only homotopy theory (and the existence of a
ring spectrum tmf with $H^*(\tmf)\approx A/\!/A_2$), while \cite[4.1]{MR} uses the Weierstrass curve.

A schematic of $\pi_*(\TMF(3))$ from the ASS viewpoint is given in Diagram \ref{TMF3diag}.
Each collection of four closely-spaced towers represents infinitely many such towers in the same stem. If the lowest of
these begins in filtration $s$, then there are such towers in filtration $s+2i$ for all $i\ge0$, with a slight exception in dimension 24.
The names of the bottom generators
are $1$, $2v_1^2$, $v_1v_2$, $2v_2^2$, $v_1^2v_2^2$, $2v_1v_2^3$, $2v_2^4$, $2v_1^2v_2^4$, $v_1v_2^5$, $2v_2^6$,
$v_1^2v_2^6$, and $2v_1v_2^7$. The name of the generator in filtration $s+2i$ is $v_1^{3i}v_2^{-i}$ times that of the
bottom generator, except that in dimension
24, we have $2v_2^4$ and $v_1^{3i}v_2^{4-i}$ for all $i>0$.
The eight  $\zt$'s along the bottom, indicated by a solid dot, occur only once, in the indicated filtration.
Because of period 48, this is the whole picture.

\medskip
\begin{minipage}{6.5in}
\begin{diag}\label{TMF3diag}{Schematic of $\pi_*(\TMF(3))$}
\begin{center}
\begin{picture}(480,200)(30,0)
\def\mp{\multiput}
\def\elt{\circle*{1.75}}
\def\ldt{\circle*{1.2}}
\put(-5,10){\line(1,0){500}}
\put(-2,0){$0$}
\put(38,0){$4$}
\put(78,0){$8$}
\put(116,0){$12$}
\put(156,0){$16$}
\put(195,0){$20$}
\put(235,0){$24$}
\put(275,0){$28$}
\put(315,0){$32$}
\put(355,0){$36$}
\put(395,0){$40$}
\put(435,0){$44$}
\mp(0,10)(1.5,20){4}{\vector(0,1){50}}
\mp(40,40)(1.5,20){4}{\vector(0,1){50}}
\mp(80,30)(1.5,20){4}{\vector(0,1){50}}
\mp(120,40)(1.5,20){4}{\vector(0,1){50}}
\mp(160,50)(1.5,20){4}{\vector(0,1){50}}
\mp(200,60)(1.5,20){4}{\vector(0,1){50}}
\put(240,60){\vector(0,1){40}}
\mp(241.5,70)(1.5,20){3}{\vector(0,1){50}}
\mp(280,80)(1.5,20){4}{\vector(0,1){50}}
\mp(320,70)(1.5,20){4}{\vector(0,1){50}}
\mp(360,80)(1.5,20){4}{\vector(0,1){50}}
\mp(400,90)(1.5,20){4}{\vector(0,1){50}}
\mp(440,100)(1.5,20){4}{\vector(0,1){50}}
\mp(7.5,116)(0,4){3}{\ldt}
\mp(47.5,146)(0,4){3}{\ldt}
\mp(87.5,136)(0,4){3}{\ldt}
\mp(127.5,146)(0,4){3}{\ldt}
\mp(167.5,156)(0,4){3}{\ldt}
\mp(207.5,166)(0,4){3}{\ldt}
\mp(247.5,156)(0,4){3}{\ldt}
\mp(287.5,186)(0,4){3}{\ldt}
\mp(327.5,176)(0,4){3}{\ldt}
\mp(367.5,186)(0,4){3}{\ldt}
\mp(407.5,196)(0,4){3}{\ldt}
\mp(447.5,206)(0,4){3}{\ldt}
\mp(0,10)(1.5,20){4}{\line(1,1){20}}
\mp(80,30)(1.5,20){4}{\line(1,1){20}}
\mp(160,50)(1.5,20){4}{\line(1,1){20}}
\mp(320,70)(1.5,20){4}{\line(1,1){20}}
\mp(400,90)(1.5,20){4}{\line(1,1){20}}
\put(250,60){\line(1,1){10}}
\mp(241.5,70)(1.5,20){3}{\line(1,1){20}}
\put(0,10){\line(3,1){60}}
\mp(30,20)(30,10){2}{\elt}
\mp(170,40)(10,10){2}{\elt}
\put(170,40){\line(1,1){10}}
\put(170,40){\line(3,1){30}}
\put(200,50){\elt}
\mp(340,70)(30,10){2}{\elt}
\put(340,70){\line(3,1){30}}
\put(420,90){\elt}
\end{picture}
\end{center}
\end{diag}
\end{minipage}

\bigskip
The extension of $x$ over $Z$ occurs because $2x=0$ and $\pi_i(\TMF(3))=0$
for $i=19$, $23$, $27$, $29$, and $30$, showing that the obstructions to extending over
the remaining cells are all 0.
\end{proof}

The spectrum $Z\w\tmf$ will be one of our connective models of $\TMF(3)$.
The following result gives its homotopy groups, which are closely related
to those of $X\w\tmf$.
\begin{thm} \label{piZ} There is an isomorphism of graded abelian groups
$$\pi_*(Z\w\tmf)\approx {\widetilde K}\oplus\zt[v_2^8]\langle \,x,\,\eta x,\,\nu x,\,x^2,\,\nu x^2,\,v_1v_2x^2,\,v_2^8 \nu,\,v_2^8\nu^2\rangle,$$
where
$${\widetilde K}=\ker({\widetilde R}\to\zt[v_2^8]\langle v_2^4\rangle )$$
with $\widetilde R$ the subgroup of the ring $R$ of \ref{v1v2thm} spanned by all elements divisible by $v_2^3$.
\end{thm}

In dimension $\le 51$, $\pi_*(Z\w\tmf)$ may be seen in Diagram \ref{bigdiag} by removing the first two $\zt$'s, and the $bo_*$ starting in
the 5-stem, and the $bsp_*$ starting in the 9-stem, and increasing stems of all elements by 3. Thus the first element would be the
$\zt$ class $x$, which appears in \ref{bigdiag} in position $(14,2)$, and is in the 17-stem for $Z$.  For the ASS-type chart
that we will describe in our proof, filtrations should be decreased by 2, so that $x$ appears in filtration 0.

\begin{proof} Let $M_7=H^*(X_7)$ be the $A$-module (or $A_2$-module) whose only nonzero groups are $\zt$ in dimensions 0, 4, 6, and 7 with $\sq^7\ne0$,
and let $M_{421}$ be the $A$-module or $A_2$-module whose only nonzero groups are $\zt$ in dimensions 0, 1, 3, and 7 with $\sq^4\sq^2\sq^1\ne0$.
There is an exact sequence
\begin{eqnarray}\nonumber&\to&\ext_{A_2}^{s-2,t-1}(\Sigma^{24}M_7)\to\ext_{A_2}^{s,t}(\Sigma^{17}M_{421})\to E_2^{s,t}(Z\w\tmf)\\
&\to&\ext_{A_2}^{s-1,t-1}(\Sigma^{24}M_7)
\mapright{d}
\ext_{A_2}^{s+1,t}(\Sigma^{17}M_{421})\to,\label{Zseq}\end{eqnarray}
with $d(\iota_{24})=h_2^2\iota_{17}$. Here $E_2(Z\w\tmf)$ is the $E_2$-term of a spectral sequence converging to $\pi_*(Z\w\tmf)$. We could compute
$E_2(Z\w\tmf)$ by first computing $\ext_{A_2}(M_7)$ and $\ext_{A_2}(M_{421})$ (and these have been computed in \cite{DM} and \cite{GM}),
but we prefer the following method which relates it directly to $E_2(X\w\tmf)$.

Let $P=\ker(d_1)$ in the resolution in the proof of Theorem \ref{extCthm}. One easily verifies that there is an exact sequence
of $A_2$-modules
$$0\to\Sigma^{24}M_7\mapright{i}\Sigma^{11}A_2/(\sq^1,\sq^5)\mapright{d_2}P\mapright{q}\Sigma^{16}M_{421}\to 0$$
with $d_2(\iota_{11})=\sq^7I_4$, $q(\sq^{6,6+7,5}I_4+\sq^{4,6}I_6)=\text{gen}_{16}$, and $i(\iota_{24})=\sq^{6,7+4,6,3}\iota_{11}$.

Note that $\ext_{A_2}(P)$ consists of a shifted version of Diagram \ref{bigdiag} minus the first two $\zt$'s and the first $bo_*$.
It is shifted so that the (now) initial tower, which did begin in $(9,2)$, now begins in $(11,0)$. Note also that \begin{equation}\label{d2s}\ext_{A_2}(P)\mapright{d_2^*}
\ext_{A_2}(\Sigma^{11}A_2/(\sq^1,\sq^5))\end{equation} is surjective, because of the $bsp_*$ in \ref{bigdiag} beginning in $(9,2)$.

Let $K=\im(d_2)=\ker(q)$. There is a commutative diagram of exact sequences, with $\ext=\ext_{A_2}$ and all Ext groups having the same second
superscript $t$,

\centerline{\xymatrix{
0\ar[d] &&& 0\ar[d]\\
{\ext^{s-2}(\Sigma^{24}M_7)}\ar[d]\ar[dr]^{\delta} && &{\ext^{s-1}(\Sigma^{24}M_7)}\ar[d]\\
{\ext^{s-1}(K)}\ar[d]\ar[r]&{\ext^s(\Sigma^{16}M_{421})}\ar[r]&{\ext^s(P)}\ar[r]^j\ar[dr]_{d_2^s}&{\ext^s(K)}\ar[d]\\
{\ext^{s-1}(\Sigma^{11}A_2/(\sq^1,\sq^5))}\ar[d]&&&{\ext^s(\Sigma^{11}A_2/(\sq^1,\sq^5))}\ar[d]\\
0&&&0
}}

\bigskip
\noindent in which $d_2^*$ is surjective. By a diagram chase, this implies exactness of
\begin{equation}\label{Zseq2}\to\ext^{s-2}(\Sigma^{24}M_7)\mapright{\delta}\ext^s(\Sigma^{16}M_{421})\to\ker(d_2^s)\to\ext^{s-1}(\Sigma^{24}M_7)
\to.\end{equation}
This $\delta$ must send $\iota_{24}$ to $h_2^2\iota_{16}$ since $\ext^{2,24}(P)=0$. Thus it must agree totally with $d$ of (\ref{Zseq}),
and so the exact sequences (\ref{Zseq}) and (\ref{Zseq2}) are identical.
Therefore, $E_2(Z\w\tmf)\approx\ker(d_2^*)$, and this is the chart obtained from \ref{bigdiag}, extended indefinitely, by removing the first two
dots, the initial $bo_*$, and the $bsp_*$ starting in $(9,2)$, and regrading so that the $\zt$ in $(14,2)$ in \ref{bigdiag} is now in $(17,0)$.
\end{proof}

\begin{cor} There is a map $Z\w\tmf\to X\w\tmf$ such that the induced map $v_2^{-1}Z\w\tmf\to v_2^{-1}X\w\tmf$ is an equivalence.\label{equcor}
\end{cor}
\begin{proof} There is a map $Z\to X\w\tmf$ extending $x$ for the same reason as in the proof of Proposition \ref{toTMF3}, namely
0 obstructions.
Smashing with $\tmf$ and following by the multiplication of $\tmf$ yields the desired map. The proof of \ref{piZ} identified
$\pi_*(Z\w\tmf)$ with the kernel of (\ref{d2s}), which is contained in $\pi_*(X\w\tmf)$. Thus $\pi_*(Z\w\tmf)$ injects
into all of $\pi_*(X\w\tmf)$ except $\nu$, $\nu^2$, and the multiples of $v_2^iv_1^j$  for $i\le 2$. These latter classes are, for $i=0$
the $bo_*$ which is  $\coker(\pi_*(\Sigma^{-1}C)\to\pi_*(X\w\tmf))$, for $i=1$ the initial $bo_*$ in \ref{bigdiag}, and for $i=2$ the
$bsp_*$ which appears in \ref{bigdiag} to begin in $(9,2)$. Since $v_2^8$ times these classes are in the image from $\pi_*(Z\w\tmf)$,
we deduce the claim that it is an equivalence after $v_2^8$ is inverted.
\end{proof}

\begin{thm}\label{equiv}
The map $Z\to \TMF(3)$ of Proposition \ref{toTMF3} induces an equivalence $v_2^{-1}Z\w\tmf\to \TMF(3)$.\end{thm}
\begin{proof} We need a fact from topological modular forms that there is a map $$\tmf\w \TMF(3)\to \TMF(3)$$ making $\TMF(3)$ a
$\tmf$-module.
Using this and the map $Z\to\TMF(3)$, we obtain a map $Z\w\tmf\to \TMF(3)$. We will show it sends $\pi_*(Z\w\tmf)$ to elements of $\pi_*(\TMF(3))$
with the same names (as those of Theorem \ref{piZ}). In the proof of Proposition \ref{toTMF3}, we discussed how
\cite[4.1]{MR} can be interpreted to give $\pi_*(\TMF(3))$ as a $v_2$-inverted version of our Theorem \ref{v1v2thm}.
Then the same argument as was used in the proof of \ref{equcor} gives the asserted equivalence.

The class $x$ maps across by construction. We must deduce from this, by various types of naturality, that all other classes map
across. Our map is one of $\tmf_*$-modules. The relation $v_1^4x=\eta v_1^3v_2^3$ is present in both $\pi_*(Z\w\tmf)$
and $\pi_*(\TMF(3))$ (by Theorem \ref{v1v2thm} and \cite[4.1]{MR}, resp.), and
hence $\eta v_1^3v_2^3$ maps across, and then so also does $v_1^3v_2^3$. Since $16v_2^2$ is in $\tmf_*$, we deduce that all $v_1^iv_2^j$
with $i\equiv3$ mod 4 and $j$ odd map across. By the Toda bracket formula $2v_1^5v_2^3=\langle \eta^2 v_1^3v_2^3,\eta,2\rangle$, which
is valid in both $Z\w\tmf$ and $\TMF(3)$, we now have that all $v_1^iv_2^j$ with $i$ and $j$ odd map across.

In \cite[4.1]{MR}, it is noted that $\pi_{20}(S^0)\to\pi_{20}(\TMF(3))$ sends $\overline\kappa$ to $\nu x$. One can show, for example using
Yoneda products, that $\overline\kappa$ acting on $x\in\pi_{17}(Z\w\tmf)$ yields the class that we call $\nu x^2$. Thus $\nu x^2$
maps across, and hence so does $x^2$. There is a bracket formula $2v_2^6=\langle x^2,\eta,2\rangle$ in both spectra,
and so $v_2^6$
maps across. Arguing as before, we deduce that all $v_1^iv_2^j$ with $i$ and $j$ even map across. Knowing that $v_2^8$ maps across
implies the same for $\nu v_2^8$ and $\nu^2v_2^8$. We have now accounted for all of $\pi_*(Z\w\tmf)$.
\end{proof}

The following corollary is immediate from \ref{equcor} and \ref{equiv}.
\begin{cor}\label{Xcor} There is an equivalence $v_2^{-1}X\w\tmf\to \TMF(3)$.\end{cor}
Thus both $X\w\tmf$ and $Z\w\tmf$ can serve as connective models of $\TMF(3)$. We prefer $X\w\tmf$ because it is a ring spectrum
and gives a better approximation to $\pi_*(\TMF(3))$ prior to inverting $v_2$, but $Z\w\tmf$ was useful because it was so easy to get a
map from it into $\TMF(3)$.

\section{A model related to $\tmf\w\tmf$}\label{BGsec}
In this section we study a third model of $\tmf(3)$ introduced in \cite{MR}. This one is closely related to $\tmf\w\tmf$,
and we provide a proof that a plausible splitting of $\tmf\w\tmf$ does not occur. We clarify some aspects of the construction
in \cite{MR} and compute the homotopy groups.

Let $A^*=\zt[\zeta_1,\zeta_2,\ldots]$ denote the dual of the mod 2 Steenrod algebra. Here $\zeta_i=\chi(\xi_i)$, the conjugates of the
usual generators. Assign a weight $\wt$ on $A^*$ by $\wt(\zeta_i)=2^{i-1}$ and $\wt(ab)=\wt(a)+\wt(b)$. It is well-known and
easily verified that $$(A/\!/A_2)^*=\zt[\zeta_1^8,\zeta_2^4,\zeta_3^2,\zeta_4,\zeta_5,\ldots]$$
and there is a splitting as $A_2$-modules
$$(A/\!/A_2)^*\approx\bigoplus_{n\ge0}M_n,$$
where $M_n$ is spanned by all monomials in $(A/\!/A_2)^*$ of weight $8n$. The $A$-action is given by $\zeta_i(\chi\sq)=\zeta_i+\zeta_{i-1}^2$.
Note that $H_*(\tmf)\approx (A/\!/A_2)^*$.

Similarly $H_*(bo)=(A/\!/A_1)^*$ is isomorphic to a polynomial algebra on $\zeta_1^4$, $\zeta_2^2$, and $\zeta_i$ for $i\ge 3$.
There are $bo$-Brown-Gitler spectra $bo_n$ satisfying that $H_*(bo_n)$ is the span of all monomials in $H_*(bo)$ with weight $\le 4n$.(\cite{GJM})
One easily verifies that there is an isomorphism of $A_2$-modules
$$\bigoplus\phi_n:\bigoplus_{n\ge0}H_*(\Sigma^{8n}bo_n)\to H_*(\tmf)$$
defined by $\phi_n(\sigma^{8n}\zeta_1^{i_1}\zeta_2^{i_2}\cdots)=\zeta_1^{8n-\sum2^ji_j}\zeta_2^{i_1}\zeta_3^{i_2}\cdots$. The image of $\phi_n$
is $M_n$, the span of monomials of weight $8n$. One might ask if this isomorphism is induced by an equivalence of the spectra
$\tmf\w\tmf$ and $\bigvee\Sigma^{8n}bo_n\w\tmf$. An analogous equivalence $bo\w bo\simeq\bigvee\Sigma^{4n}{\overline B}_n\w bo$ was proved in
\cite{bo}. In that case ${\overline B}_n$ was an integral Brown-Gitler spectrum.

We answer this question and prepare for a new $\TMF(3)$ model  by proving the following result.
\begin{thm}\label{nosplit} The spectra $\tmf\w\tmf$ and $\bigvee\limits_{n\ge0}\Sigma^{8n}bo_n\w\tmf$ are not homotopy equivalent.
Indeed, in the ASS converging to $\pi_*(\tmf\w\tmf)$, which has $$E_2\approx\bigoplus_{n\ge0}\ext_{A_2}(H^*(\Sigma^{8n}bo_n)),$$
there is a class $g\in\ext_{A_2}^{0,24}(H^*(\Sigma^{16}bo_2))$ and an element $w\in\ext_{A_2}^{3,26}(H^*(\Sigma^8 bo_1))$
such that $d_3(g)=w$.
\end{thm}
\begin{proof} Let $\tmfbar$ denote the cofiber of $S^0\to\tmf$. Since $\tmf$ is a ring spectrum, there is a splitting
$$\tmf\w\tmf\simeq (S^0\w\tmf)\vee(\tmfbar\w\tmf).$$
We will use the cofibration
\begin{equation}\label{tmfcof}\tmfbar\w S^0\to\tmfbar\w\tmf\to\tmfbar\w\tmfbar\end{equation}
and a differential in the ASS of $\tmfbar$ to deduce the claimed differential in the ASS of $\tmfbar\w\tmf$.

In Diagram \ref{bo1bo2diag}, we depict $\ext_{A_2}^{s,t}(H^*(\Sigma^8bo_1\vee\Sigma^{16}bo_2))$ for $s<8$, $t-s<40$.
Elements suggested by solid dots come from the first summand, and those with open circles (or connected to open circles by lines)
come from the second summand.

\medskip
\begin{minipage}{6.5in}
\begin{diag}\label{bo1bo2diag}{$\ext_{A_2}^{s,t}(H^*(\Sigma^8bo_1\vee\Sigma^{16}bo_2))$ in a range}
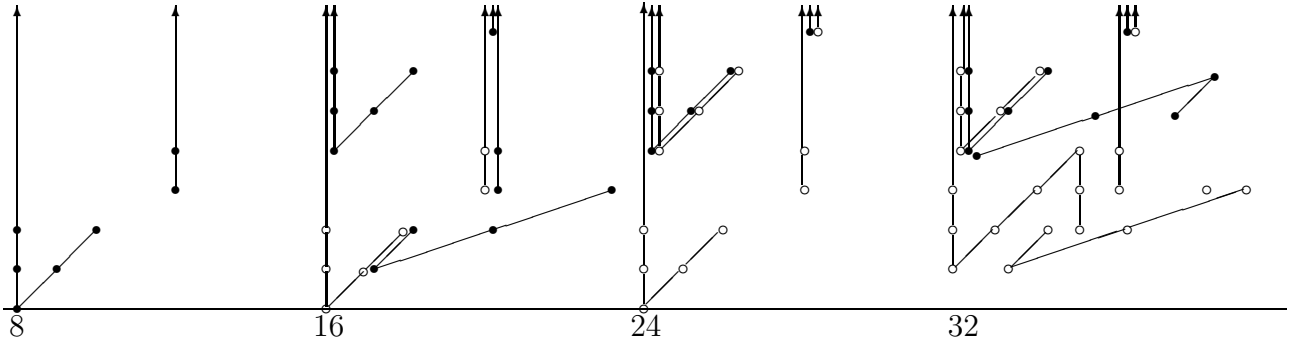
\begin{center}
\begin{picture}(480,140)(40,0)
\def\mp{\multiput}
\def\elt{\circle*{3}}
\def\cir{\circle{3}}
\put(-5,15){\line(1,0){485}}
\put(-3,5){$8$}
\put(112,5){$16$}
\put(232,5){$24$}
\put(352,5){$32$}
\put(0,15){\vector(0,1){115}}
\put(0,15){\line(1,1){30}}
\mp(0,15)(0,15){3}{\elt}
\mp(15,30)(15,15){2}{\elt}
\put(60,60){\vector(0,1){70}}
\mp(60,60)(0,15){2}{\elt}
\put(120,75){\vector(0,1){55}}
\put(120,75){\line(1,1){30}}
\mp(120,75)(0,15){3}{\elt}
\mp(135,90)(15,15){2}{\elt}
\put(240,75){\vector(0,1){55}}
\put(240,75){\line(1,1){30}}
\mp(240,75)(0,15){3}{\elt}
\mp(255,90)(15,15){2}{\elt}
\mp(243,75)(0,15){3}{\cir}
\mp(258,90)(15,15){2}{\cir}
\mp(243,77)(0,15){2}{\line(0,1){11}}
\mp(244,76)(15,15){2}{\line(1,1){13}}
\put(243,107){\vector(0,1){23}}
\mp(303,120)(120,0){2}{\cir}
\mp(303,122)(120,0){2}{\vector(0,1){8}}
\put(360,75){\vector(0,1){55}}
\put(360,75){\line(1,1){30}}
\mp(360,75)(0,15){3}{\elt}
\mp(375,90)(15,15){2}{\elt}
\mp(180,120)(120,0){3}{\vector(0,1){10}}
\mp(180,120)(120,0){3}{\elt}
\put(182,60){\vector(0,1){70}}
\mp(182,60)(0,15){2}{\elt}
\mp(135,30)(15,15){2}{\elt}
\put(135,30){\line(1,1){15}}
\put(135,30){\line(3,1){90}}
\mp(180,45)(45,15){2}{\elt}
\mp(363,73)(45,15){3}{\elt}
\put(363,73){\line(3,1){90}}
\put(438,88){\elt}
\put(438,88){\line(1,1){15}}
\mp(117,15)(0,15){3}{\cir}
\mp(131,29)(15,15){2}{\cir}
\mp(118,16)(14,14){2}{\line(1,1){13}}
\mp(117,16)(0,15){2}{\line(0,1){13}}
\put(117,46){\vector(0,1){84}}
\mp(177,60)(0,15){2}{\cir}
\mp(177,62)(120,0){3}{\line(0,1){11}}
\mp(177,76)(120,0){3}{\vector(0,1){54}}
\mp(237,15)(0,15){3}{\cir}
\mp(252,30)(15,15){2}{\cir}
\mp(238,17)(15,15){2}{\line(1,1){12}}
\mp(237,17)(0,15){2}{\line(0,1){11}}
\put(237,47){\vector(0,1){84}}
\mp(298,60)(0,15){2}{\cir}
\mp(357,75)(0,15){3}{\cir}
\mp(372,90)(15,15){2}{\cir}
\mp(357,77)(0,15){2}{\line(0,1){11}}
\mp(358,76)(15,15){2}{\line(1,1){12}}
\put(358,106){\vector(0,1){24}}
\mp(354,30)(0,15){3}{\cir}
\mp(370,45)(16,15){3}{\cir}
\mp(355,31)(16,15){3}{\line(1,1){14}}
\mp(354,32)(0,15){2}{\line(0,1){11}}
\put(354,62){\vector(0,1){68}}
\mp(417,60)(0,15){2}{\cir}
\mp(402,45)(0,15){2}{\cir}
\mp(402,47)(0,15){2}{\line(0,1){11}}
\mp(375,30)(45,15){3}{\cir}
\put(390,45){\cir}
\put(376,31){\line(1,1){13}}
\mp(376,31)(45,15){2}{\line(3,1){43}}
\put(450,60){\cir}
\end{picture}
\end{center}
\end{diag}
\end{minipage}

\bigskip

The cofibration which defines $\tmfbar$ induces an exact sequence
$$\to\ext_A^{s,t}(H^*(\tmf))\to\ext_A^{s,t}(H^*(\tmfbar))\to\ext_A^{s+1,t}(H^*(S^0))\to\ext^{s+1,t}_A(H^*(\tmf))\to.$$
There is a lower vanishing line in $\ext_A(H^*(\tmf))\approx\ext_{A_2}(\zt)$ (e.g. \cite[2.6]{DM}) which implies
that $\ext_A^{s,t}(H^*(\tmfbar))\approx\ext_A^{s+1,t}(H^*(S^0))$ if $s\le6$ and $t-s\ge31$.
In \cite{Brun}, it was shown that in the ASS of $S^0$ there are nonzero elements $e_1\in\ext^{4,42}_A(H^*(S^0))$ and $h_1t\in\ext_A^{7,44}(H^*(S^0))$
satisfying $d_3(e_1)=h_1t$. These elements are in the range of our asserted isomorphism, and so there must be corresponding elements
 $\overline{e_1}\in\ext_A^{3,42}(H^*(\tmfbar))$ and $\overline{h_1t}\in\ext_A^{6,44}(H^*(\tmfbar))$ related by a $d_3$-differential.

Now we consider the exact sequences in $\ext_A(-)$ and $\pi_*(-)$ induced by (\ref{tmfcof}). Using Bruner's software, we see
that $\ext_A^{s,t}(H^*(\tmfbar\w\tmfbar))=0$ if $t-s=39$ and $s>3$. Thus neither of the elements $\overline{e_1}$ or $\overline{h_1t}$
can be in the image from $\ext_A(H^*(\tmfbar\w\tmfbar))$, the second since there is nothing to hit it, and the first since
a class which hits it would have to support a differential, but there is nothing for it to hit. Thus the elements
$\overline{e_1}$ and $\overline{h_1t}$ related by the $d_3$ in the ASS of $\tmfbar$ map nontrivially to $\ext_A(H^*(\tmfbar\w\tmf))$.
One easily checks that $\ext_{A_2}(H^*(\Sigma^{24}bo_3)\oplus H^*(\Sigma^{32}bo_4))$ is 0 in these bigradings. Thus the
elements $\overline{e_1}$ and $\overline{h_1t}$ must map nontrivially to classes in $\ext_{A_2}(H^*(\Sigma^8bo_1)\oplus H^*(\Sigma^{16}bo_2))$
involved in a $d_3$-differential. These must be the two classes at the extreme right end of Diagram \ref{bo1bo2diag}, one in filtration 6 from $\Sigma^8bo_1$
and the other in filtration 3 from $\Sigma^{16}bo_2$.

This already implies the first conclusion of the theorem, that $\tmf\w\tmf$ does not split as $\bigvee\limits_{n\ge0}\Sigma^{8n}bo_n\w\tmf$.
We would like to infer from this differential the claimed nontrivial $d_3$ on the class $g$ in position $(24,0)$.
Clearly the $h_2$-action and the nonzero $d_3$ from $(39,3)$
implies that $d_3$ is nonzero on the class in $(33,1)$. Let $X_7=S^0\cup_\nu e^4\cup_\eta e^6\cup_2 e^7$ as before.
If $d_3(g)=0$, then
the homotopy class $g$ would extend to a map $\Sigma^{24}X_7\to \tmfbar\w\tmf$, since Diagram \ref{bo1bo2diag}
shows that there are no
obstructions to the extension. Smashing with $\tmf$ and following by the multiplication of $\tmf$ would yield a map
$\Sigma^{24}X_7\w\tmf \to\tmfbar\w\tmf$
extending $g$. Since $X_7=bo_1$, the ASS of $\Sigma^8X_7\w\tmf$ is just the black elements in Diagram \ref{bo1bo2diag}. The 16-suspension of the element
in $(17,1)$ in that diagram does not support a differential in $\Sigma^{24}X_7\w\tmf$ but would map to the class in $(33,1)$ in $\tmfbar\w\tmf$
which we showed does support a differential.
This contradicts the assumption that $d_3(g)=0$.

\end{proof}

Now we begin working toward the construction of our third connective model of $\TMF(3)$.
\begin{prop} \label{Wprop} There is a subcomplex $W_1$ of $\tmfbar$ such that there is a cofibration
$$\Sigma^8bo_1\to W_1\to\Sigma^{16}bo_2$$
which has a short exact sequence in mod-2 cohomology.\end{prop}
\begin{proof} We use the description of $H_*(\tmfbar)$ given in the second paragraph of this section. All elements of weight $\le16$
are in dimension $\le31$, and the first few elements of weight greater than 16 are $\zeta_1^{24}$, $\zeta_1^{16}\zeta_2^4$, $\zeta_1^{16}\zeta_3^2$,
and $\zeta_1^{16}\zeta_4$. The $A$-module structure of $H^*(\tmfbar^{(31)}/\tmfbar^{(23)})$ is
\begin{equation}\label{spplit}\langle \sq^0,\,\sq^2,\,\sq^3,\,\sq^4,\,\sq^5,\,\sq^6,\,\sq^7\rangle \widehat{\zeta_2^8}\ \oplus\,\langle\sq^0,\,\sq^4,\,\sq^6,\,\sq^7\rangle
\widehat{\zeta_1^{24}},\end{equation}
with the first (resp.~second) summand dual to monomials of weight 16 (resp.~24).
Here the $\widehat{(\quad)}$ represents duality.
Bruner's software shows that there is a map $$\tmfbar^{(31)}/\tmfbar^{(23)}\to \Sigma^{24}X_7$$ which induces the identity
homomorphism from the second summand of (\ref{spplit}) and 0 from the first. This is done by computing $\ext_A$ of the tensor product of
the dual of the module in (\ref{spplit}) with $M_7$, and seeing that there are no possible differentials from the obvious filtration-0 class.
 The desired complex $W_1$ is the fiber of the composite
$$\tmfbar^{(31)}\to \tmfbar^{(31)}/\tmfbar^{(23)}\to \Sigma^{24}X_7,$$
where the second map is the one just noted.\end{proof}

The $E_2$-term of the ASS for $W_1\w\tmf$ in dimension less than 40 is given in Diagram \ref{bo1bo2diag}, and, as established in Theorem
\ref{nosplit}, there are $d_3$-differentials on the classes in positions $(24,0)$, $(33,1)$, $(36,2)$, and $(39,3)$.
Let $f:S^{32}\to W_1\w\tmf$ be a nontrivial map of Adams filtration 1, which exists by Diagram \ref{bo1bo2diag}.
Smash with $\tmf$ and follow by
the multiplication of $\tmf$, obtaining a map $S^{32}\w\tmf\to W_1\w\tmf$.
\begin{defin}
Define $W$ to be the cofiber of this map $S^{32}\w\tmf\to W_1\w\tmf$.\end{defin}
This $W$ will be our third connective model of $\TMF(3)$. Note that, unlike the first two, it is not obtained as the smash product
of a finite complex with $\tmf$, since the above map $f$ does not factor through $W_1$ itself.

Similarly, let $S^{16}\to bo_2\w\tmf$ correspond to essentially the same class, as the open circles in Diagram \ref{bo1bo2diag}
depict the ASS of $\Sigma^{16}bo_2$. Extend this to a map $S^{16}\w\tmf\to bo_2\w\tmf$, and let $\widetilde{bo_2}$ denote the cofiber of this.
There is a cofiber sequence
\begin{equation}\label{Wcof} \Sigma^8bo_1\w\tmf\to W\to\Sigma^{16}\bot.\end{equation}

The short exact sequence of $A$-modules
$$0\to \Sigma^{17}A/\!/A_2\to H^*(\bot)\to A\otimes_{A_2}H^*(bo_2)\to 0$$
induces an exact sequence in $\ext_A$ which implies that $\ext_A(H^*(\bot))$ begins as the 16-desuspension of the open
circles in Diagram \ref{bo1bo2diag} with the portion connected to the element in $(32,1)$ removed. It contains no unpictured
elements in filtration 0 or 1. Therefore,
 $H^*(\bot)=A\otimes_{A_2}B$, where $B$ sits in a short exact sequence of $A_2$-modules
\begin{equation}0\to\Sigma^{17}\zt\to B\to H^*(bo_2)\to 0,\label{Bseq}\end{equation}
with the new class in $B$ equal to $\sq^4\sq^6\sq^7\iota_0$, or equivalently $\sq^4\sq^2\sq^3\iota_8$. It also equals $\sq^2$ of the top class of
$H^*(bo_2)$. The $A_2$-module $B$ cannot be given the structure of $A$-module, as the Adem relation $\sq^2\sq^{15}=\sq^1\sq^{16}+\sq^{16}\sq^1$
would be violated.

Our next result gives a direct relationship among $\ext_{A_2}(A_2/(\sq^4,\sq^{5,1}))$, which was depicted through degree 48 in Diagram \ref{bigdiag}
and is very closely related to the homotopy groups described in Theorem \ref{v1v2thm}, and $\ext_{A_2}(B)$ and $\ext_{A_2}(H^*(X_7))$, which two
together are related to the ASS of $W$. After stating and proving this result, we will use it to determine $\pi_*(W)$ and see that
$v_2^{-1}W$ is another model for $\TMF(3)$.

We begin by noting that $\ext^{s,t}_{A_2}(A_2(\sq^4,\sq^{5,1}))\approx\ext_{A_2}^{s+1,t}(A_2/(\sq^4,\sq^{5,1}))$.
\begin{thm}\label{algthm} Let $\extt_{A_2}(A_2(\sq^4,\sq^{5,1}))$ denote $\ext_{A_2}(A_2(\sq^4,\sq^{5,1}))$ without the $\zt$ in $\ext^{0,4}$
or the tower beginning in $\ext^{1,11}$. There is an exact sequence
$$\to\ext_{A_2}^{s+2,t}(\Sigma^6M_7)\to\extt_{A_2}^{s+2,t}(A_2(\sq^4,\sq^{5,1}))\to\ext_{A_2}^{s,t}(\Sigma^{16}B)\to\ext_{A_2}^{s+3,t}(\Sigma^6M_7)\to.$$
\end{thm}
\begin{proof} One can verify that there is an exact sequence of $A_2$-modules
$$0\to K\mapright{i} \Sigma^4 A_2\to A_2(\sq^4,\sq^{5,1})\mapright{\phi} \Sigma^6M_7\to 0,$$
where $\Sigma^6M_7$ is generated by $\phi(\sq^{5,1})$, and $i(K)$ is the submodule of $\Sigma^4A_2$ generated
by $\sq^7\iota_4$, and that there is a short exact sequence of $A_2$-modules
$$0\to\Sigma^{16}B\to \Sigma^{11}A_2/\!/A_0\to K\to 0$$
with $B$ as above, and the $A_2$-generators of $\Sigma^{16}B$ mapping to $\sq^5\iota_{11}$ and $\sq^{4,6,3}\iota_{11}$.

Let $R=\coker(i)=\ker(\phi)$. Except for the classes omitted in forming $\extt$, we have isomorphisms
$$\ext_{A_2}^{s}(\Sigma^{16}B)\mapright{\approx}\ext_{A_2}^{s+1}(K)\mapright{\approx}\ext_{A_2}^{s+2}(R)$$
and an exact sequence
$$\to\ext_{A_2}^{s+2}(\Sigma^6M_7)\to\ext_{A_2}^{s+2}(A_2(\sq^4,\sq^{5,1}))\to\ext_{A_2}^{s+2}(R)\to\ext_{A_2}^{s+3}(\Sigma^6M_7),$$
from which the result follows.
\end{proof}

Similarly to Theorem \ref{piZ}, we can now deduce the following result without using complete information about $\ext_{A_2}(B)$.
\begin{thm}\label{piW}
There is an isomorphism of graded abelian groups
$$\pi_*(W)\approx K'\oplus\zt[v_2^8]\langle \,x,\,\eta x,\,\nu x,\,x^2,\,\nu x^2,\,v_1v_2x^2,\,v_2^8 \nu,\,v_2^8\nu^2\rangle,$$
where
$$K'=\ker(R'\to\zt[v_2^8]\langle v_2^4\rangle )$$
with $R'$ the subgroup of the ring $R$ of \ref{v1v2thm} spanned by all elements divisible by $v_2$ but not including the cyclic group
generated by $2v_2^2$.
\end{thm}
\begin{proof}
The map $\Sigma^{15}\bot\to\Sigma^8bo_1\w\tmf$ whose cofiber is $W$ has Adams filtration 3 since $H^i(\Sigma^{15}\bot)=0$ for $i<15$ and for
$i=17$, $18$, and $20$,
the values of $i$ for which $\pi_i(\Sigma^8bo_1\w\tmf)$ has nonzero classes in filtration less than 3.
We obtain a homomorphism
$$\ext_{A_2}^{s,t}(\Sigma^{15}B)\to\ext_{A_2}^{s+3,t+3}(\Sigma^8M_7).$$
We claim that this is the same homomorphism as the one at the end of the exact sequence in Theorem \ref{algthm}.

Both of them are nontrivial on the class in $\ext_{A_2}^{0,24}(\Sigma^{16}B))$, the first by Theorem \ref{nosplit} and the second
since Diagram \ref{bigdiag} is 0 in position $(21,3)$.
Let $\bc$ (resp.~$\bd$) be a minimal $A_2$-resolution of $\Sigma^8M_7$ (resp.~$\Sigma^{15}B$). There is a morphism
$C_3\mapright{\phi} \Sigma^{15}B$ which lifts to a morphism $C_3\to D_0$ and then to $C_{s+3}\to D_s$ for all $s$.
Since $B_5=0$, $\phi$ must be 0 on the generators in 8, 12, and 20, and it must send the generator in 23 nontrivially to get the
correct Ext morphism. This completely determines the entire Ext morphism. The same is true of the Ext morphism at the end
of the sequence of \ref{algthm}. Thus the two Ext morphisms are equal.

We obtain that $E_2^{s,t}(W)\approx\extt^{s,t-2}_{A_2}(A_2(\sq^4,\sq^{5,1}))$. We have already seen that there are no possible
differentials in an ASS with $E_2\approx \extt_{A_2}(A_2(\sq^4,\sq^{5,1}))$. Thus $\pi_*(W)$ is like the groups described in Theorem
\ref{v1v2thm} without the initial $bo_*$, $\nu$, $\nu^2$, or the $2v_2^2$-tower.
\end{proof}

Similarly to Corollary \ref{Xcor}, we obtain the following result, giving a third connective model for $\TMF(3)$.
The significance of this one is its close relationship to $\tmf$.
\begin{cor}\label{Wcor} There is an equivalence $v_2^{-1}W\to \TMF(3)$.\end{cor}
\begin{proof} Similarly to \ref{equcor}, we construct a map $Z\to W$, then use the $\tmf$-module structure of $W$ to
extend to a map $Z\w\tmf\to W$. This becomes an equivalence after inverting $v_2$. Then we use Theorem \ref{equiv}.\end{proof}

\section{$\tmf(3)$-homology of real projective space}\label{Psec}
In this section, we compute $\pi_*(X\w\tmf\w P_1)$, where $X$ is as in Theorem \ref{GMthm} and $P_1=RP^\infty$. Because $X\w\tmf$
is probably the best connective model for $\TMF(3)$, this could be considered as $\tmf(3)_*(P_1)$. More work will be required to deduce
results for $P_n$ or $P_1^m$ from this, but this should provide a model. One possible application of this calculation would be
to obstruction theory, which was an initial motivation for this project.

It is convenient to state and prove the result for $\Sigma P_1$. Some of the $\tmf_*$-module structure
is included in the result. We now state the main theorem of this section. Although it is not exactly an ASS, we describe the
groups in an ASS-like way, with bigrading $(i,s)$ for an element of $\pi_i(X\w\tmf\w \Sigma P_1)$ of filtration $s$.
Many elements are expressed as $a^{e_1}v_2^{e_2}$ of bigrading $(2e_1+6e_2,e_2)$. Thus $a$ (resp.~$v_2$) is thought of as having
bigrading $(2,0)$ (resp.~$(6,1)$), although $a$ and $v_2$ themselves are not actually elements of $\pi_*(X\w\tmf\w\Sigma P_1)$.
Certain powers of $v_2$ can be thought of as being part of the $\tmf_*$-module structure.
Note that the elements $a^{e_1}v_2^{e_2}$ are not really products, since $X\w\tmf\w \Sigma P_1$ is not a ring spectrum.
The element $a$ roughly corresponds to
$v_1/2$.

\begin{thm} \label{P1thm} For each pair $(e_1,e_2)$ such that $e_1>0$, $e_2\ge0$, and $e_1\equiv e_2\ (2)$, $\pi_*(X\w\tmf\w\Sigma P_1)$ has a summand $\Z/2^{e_1}$
generated by
\begin{equation}\label{cas}\begin{cases} a^{e_1}v_2^{e_2}&\text{if }e_1\equiv e_2\ (4)\\
2a^{e_1}v_2^{e_2}&\text{if }e_1\equiv e_2+2\ (4),\end{cases}\end{equation}
with the following two variations:
\begin{itemize}
\item if $e_1=2$ and $e_2\equiv0\ (8)$, it is $\Z/8$ generated by $a^2v_2^{e_2}$;
\item if $e_1=1$ and $e_2\equiv1\text{ or }3\ (8)$, it is $\Z/4$ generated by $av_2^{e_2}$.
\end{itemize}
If $e_1\ge5$ and $e_1\equiv e_2\ (4)$, or if $(e_1,(e_2\text{ mod }8))=(4,0)$ or $(3,3)$, then $\eta^2 a^{e_1}v_2^{e_2}\ne0$.
If $(e_1,(e_2\text{ mod }8))=(1,1)$, $(4,4)$, $(2,6)$, or $(3,7)$, then $\eta a^{e_1}v_2^{e_2}\ne0$.

If $e_1\ge3$ and $e_1\equiv e_2+2\ (4)$, or $e_1=2$ and $e_2\equiv0\ (8)$, then there exists $b_{e_1,e_2}$ of bigrading
$(e_1+e_2-2,2e_1+6e_2-2)$ and order 2 satisfying $\eta^2b_{e_1,e_2}=2^{e_1}a^{e_1}v_2^{e_2}$. If $(e_1,(e_2\text{ mod }8))=(1,3)$
or $(2,4)$, there exists $b'_{e_1,e_2}$ of bigrading $(e_1+e_2-1,2e_1+6e_2-1)$ and order 2 satisfying $\eta b'_{e_1,e_2}
=2^{e_1}a^{e_1}v_2^{e_2}$.

In addition, there are the following $\zt$ classes $x_{i,s}$ of bigrading $(i,s)$.\footnote{Note that the subscripts of
$x$ refer to bigrading, while the subscripts of $b$ and $b'$ do not.}
\begin{itemize}
\item $x_{8i+2,1}$ for $i\ge1$.

All the rest are acted on freely by $v_2^8$.

\item $x_{5,1}=\nu b_{2,0}$,\ $x_{7,1}=\nu a^2$;
\item $x_{6,1}$ satisfying $\nu x_{6,1}=\eta av_2$;
\item $x_{21,3}$ and $\nu x_{21,3}$;
\item $x_{22,4}=\nu b'_{1,3}$, \ $x_{23,4}=\nu av_2^3$;
\item $x_{36,6}$ and $\nu x_{36,6}$,\ $x_{37,6}$ and $\nu x_{37,6}$;
\item $x_{38,6}$ satisfying $\nu x_{38,6}=\eta a^2v_2^6$.
\end{itemize}
\end{thm}

In Diagrams \ref{P1diag1} and \ref{P1diag2} we depict the groups of Theorem \ref{P1thm}.
 All elements except
those in position $(8i+2,1)$ for $i\ge1$ in Diagram \ref{P1diag1} are acted on freely by $v_2^8$.

\medskip
\begin{minipage}{6.5in}
\begin{diag}\label{P1diag1}{$\pi_*(X\w\tmf\w \Sigma P_1)$ in $*<32$}
\begin{center}
\begin{picture}(480,280)(80,0)
\def\mp{\multiput}
\def\elt{\circle*{3}}
\def\cir{\circle{3}}
\put(0,18){\line(1,1){36}}
\mp(0,18)(18,18){3}{\elt}
\put(36,18){\line(0,1){36}}
\mp(36,18)(0,18){3}{\elt}
\put(108,18){\line(0,1){54}}
\mp(108,18)(0,18){4}{\elt}
\put(108,18){\line(1,1){36}}
\mp(126,36)(18,18){2}{\elt}
\put(36,18){\line(3,1){108}}
\put(90,36){\elt}
\put(111,36){\line(0,1){18}}
\put(111,36){\line(1,1){17}}
\mp(111,36)(17,17){2}{\elt}
\put(111,54){\elt}
\put(74,35){\elt}
\put(74,35){\line(3,1){54}}
\put(2,16){\line(3,1){108}}
\put(56,34){\elt}
\put(144,36){\elt}
\put(147,54){\line(1,1){32}}
\mp(147,54)(16,16){3}{\elt}
\put(179,54){\line(0,1){32}}
\mp(179,54)(0,16){3}{\elt}
\put(182,36){\line(0,1){90}}
\mp(182,36)(0,18){6}{\elt}
\put(146,90){\line(1,1){36}}
\mp(146,90)(18,18){2}{\elt}
\put(249,18){\line(0,1){126}}
\mp(249,18)(0,18){8}{\elt}
\put(249,18){\line(1,1){36}}
\mp(267,36)(18,18){2}{\elt}
\put(285,36){\elt}
\put(252,36){\line(0,1){72}}
\mp(252,36)(0,18){5}{\elt}
\put(252,36){\line(1,1){34}}
\mp(269,53)(17,17){2}{\elt}
\put(255,54){\line(0,1){18}}
\mp(255,54)(0,18){2}{\elt}
\put(329,36){\line(0,1){162}}
\mp(329,36)(0,18){10}{\elt}
\put(291,160){\line(1,1){38}}
\mp(291,160)(19,19){2}{\elt}
\put(326,54){\line(0,1){108}}
\mp(326,54)(0,18){7}{\elt}
\mp(290,126)(18,18){2}{\elt}
\put(290,126){\line(1,1){36}}
\put(323,72){\line(0,1){54}}
\mp(323,72)(0,18){4}{\elt}
\put(289,92){\line(1,1){34}}
\mp(289,92)(17,17){2}{\elt}
\put(320,72){\line(0,1){17}}
\mp(320,72)(0,17){2}{\elt}
\put(304,73){\elt}
\put(304,73){\line(1,1){16}}
\put(304,73){\line(3,1){54}}
\put(358,91){\elt}
\mp(345,74)(42,14){2}{\elt}
\put(345,74){\line(3,1){42}}
\put(390,18){\line(0,1){198}}
\mp(390,18)(0,18){12}{\elt}
\put(390,18){\line(1,1){38}}
\mp(409,37)(19,19){2}{\elt}
\put(428,37){\elt}
\put(393,36){\line(0,1){144}}
\mp(393,36)(0,18){9}{\elt}
\put(393,36){\line(1,1){36}}
\mp(393,36)(18,18){3}{\elt}
\put(396,54){\line(0,1){90}}
\mp(396,54)(0,18){6}{\elt}
\put(396,54){\line(1,1){34}}
\mp(396,54)(17,17){3}{\elt}
\put(399,72){\line(0,1){36}}
\mp(399,72)(0,18){3}{\elt}
\put(399,72){\line(1,1){32}}
\mp(399,72)(16,16){3}{\elt}
\put(320,72){\line(3,1){57}}
\put(377,91){\elt}
\put(376,91){\line(4,1){56}}
\put(474,36){\line(0,1){234}}
\mp(474,36)(0,18){14}{\elt}
\put(474,270){\line(-1,-1){38}}
\mp(474,270)(-19,-19){3}{\elt}
\put(471,54){\line(0,1){180}}
\mp(471,54)(0,18){11}{\elt}
\put(471,234){\line(-1,-1){36}}
\mp(471,234)(-18,-18){3}{\elt}
\put(468,72){\line(0,1){126}}
\mp(468,72)(0,18){8}{\elt}
\put(468,198){\line(-1,-1){34}}
\mp(468,198)(-17,-17){3}{\elt}
\put(465,90){\line(0,1){72}}
\mp(465,90)(0,18){5}{\elt}
\put(465,162){\line(-1,-1){32}}
\mp(465,162)(-16,-16){3}{\elt}
\put(462,108){\line(0,1){18}}
\mp(462,108)(0,18){2}{\elt}
\put(462,126){\line(-1,-1){15}}
\put(447,111){\elt}
\put(-2,18){\line(1,0){479}}
\put(-2,5){$2$}
\put(34,5){$4$}
\put(106,5){$8$}
\put(176,5){$12$}
\put(248,5){$16$}
\put(319,5){$20$}
\put(391,5){$24$}
\put(463,5){$28$}
\end{picture}
\end{center}
\end{diag}
\end{minipage}

\bigskip
\medskip
\begin{minipage}{6.5in}
\begin{diag}\label{P1diag2}{$\pi_*(X\w\tmf\w \Sigma P_1)$, $32\le*<48$}
\begin{center}
\begin{picture}(400,470)(80,0)
\def\mp{\multiput}
\def\elt{\circle*{3}}
\def\cir{\circle{3}}
\put(103,20){\line(0,1){300}}
\mp(103,20)(0,20){16}{\elt}
\put(103,20){\line(1,1){42}}
\mp(103,20)(21,21){3}{\elt}
\put(106,40){\line(0,1){240}}
\mp(106,40)(0,20){13}{\elt}
\put(106,40){\line(1,1){40}}
\mp(106,40)(20,20){3}{\elt}
\put(109,60){\line(0,1){180}}
\mp(109,60)(0,20){10}{\elt}
\put(109,60){\line(1,1){38}}
\mp(109,60)(19,19){3}{\elt}
\put(112,80){\line(0,1){120}}
\mp(112,80)(0,20){7}{\elt}
\put(112,80){\line(1,1){36}}
\mp(112,80)(18,18){3}{\elt}
\put(115,100){\line(0,1){60}}
\mp(115,100)(0,20){4}{\elt}
\put(115,100){\line(1,1){18}}
\put(133,118){\elt}
\put(181,140){\line(0,1){40}}
\mp(181,140)(0,20){3}{\elt}
\put(181,180){\line(-1,-1){36}}
\mp(181,180)(-18,-18){3}{\elt}
\put(184,120){\line(0,1){100}}
\mp(184,120)(0,20){6}{\elt}
\put(184,220){\line(-1,-1){38}}
\mp(184,220)(-19,-19){3}{\elt}
\put(187,100){\line(0,1){160}}
\mp(187,100)(0,20){9}{\elt}
\put(187,260){\line(-1,-1){40}}
\mp(187,260)(-20,-20){3}{\elt}
\put(190,80){\line(0,1){220}}
\mp(190,80)(0,20){12}{\elt}
\put(190,300){\line(-1,-1){42}}
\mp(190,300)(-21,-21){3}{\elt}
\put(193,60){\line(0,1){280}}
\mp(193,60)(0,20){15}{\elt}
\put(193,340){\line(-1,-1){44}}
\mp(193,340)(-22,-22){3}{\elt}
\put(196,40){\line(0,1){340}}
\mp(196,40)(0,20){18}{\elt}
\put(196,380){\line(-1,-1){46}}
\mp(196,380)(-23,-23){3}{\elt}
\put(272,20){\line(0,1){380}}
\mp(272,20)(0,20){20}{\elt}
\put(272,20){\line(1,1){44}}
\mp(272,20)(22,22){3}{\elt}
\put(275,40){\line(0,1){320}}
\mp(275,40)(0,20){17}{\elt}
\put(275,40){\line(1,1){42}}
\mp(275,40)(21,21){3}{\elt}
\put(278,60){\line(0,1){260}}
\mp(278,60)(0,20){14}{\elt}
\put(278,60){\line(1,1){40}}
\mp(278,60)(20,20){3}{\elt}
\put(281,80){\line(0,1){200}}
\mp(281,80)(0,20){11}{\elt}
\put(281,80){\line(1,1){38}}
\mp(281,80)(19,19){3}{\elt}
\put(284,100){\line(0,1){140}}
\mp(284,100)(0,20){8}{\elt}
\put(284,100){\line(1,1){36}}
\mp(284,100)(18,18){3}{\elt}
\put(287,120){\line(0,1){80}}
\mp(287,120)(0,20){5}{\elt}
\put(287,120){\line(1,1){34}}
\mp(287,120)(17,17){3}{\elt}
\put(290,140){\line(0,1){20}}
\mp(290,140)(0,20){2}{\elt}
\put(290,140){\line(1,1){17}}
\mp(290,140)(17,17){2}{\elt}
\mp(199,140)(57,19){2}{\elt}
\put(199,140){\line(3,1){57}}
\mp(215,141)(54,18){2}{\elt}
\put(215,141){\line(3,1){54}}
\put(234,140){\elt}
\put(234,140){\line(4,1){72}}
\mp(145,40)(171,0){2}{\elt}
\put(321,183){\line(1,1){36}}
\mp(321,182)(18,18){2}{\elt}
\put(357,160){\line(0,1){60}}
\mp(357,160)(0,20){4}{\elt}
\put(354,180){\elt}
\mp(322,222)(19,19){2}{\elt}
\put(322,222){\line(1,1){38}}
\put(360,140){\line(0,1){120}}
\mp(360,140)(0,20){7}{\elt}
\mp(323,260)(20,20){2}{\elt}
\put(323,260){\line(1,1){40}}
\put(363,120){\line(0,1){180}}
\mp(363,120)(0,20){10}{\elt}
\mp(324,298)(21,21){2}{\elt}
\put(324,298){\line(1,1){42}}
\put(366,100){\line(0,1){240}}
\mp(366,100)(0,20){13}{\elt}
\mp(325,336)(22,22){2}{\elt}
\put(325,336){\line(1,1){44}}
\put(369,80){\line(0,1){300}}
\mp(369,80)(0,20){16}{\elt}
\mp(326,378)(23,23){2}{\elt}
\put(326,378){\line(1,1){46}}
\put(372,60){\line(0,1){364}}
\mp(372,60)(0,20){17}{\elt}
\mp(372,380)(0,22){3}{\elt}
\mp(327,418)(24,24){2}{\elt}
\put(327,418){\line(1,1){48}}
\put(375,40){\line(0,1){428}}
\mp(375,40)(0,20){18}{\elt}
\mp(375,380)(0,22){5}{\elt}
\put(100,20){\line(1,0){278}}
\put(102,6){$32$}
\put(186,6){$36$}
\put(273,6){$40$}
\put(363,6){$44$}
\end{picture}
\end{center}
\end{diag}
\end{minipage}

The remainder of this section is devoted to the proof of Theorem \ref{P1thm}. By Theorem \ref{GMthm}, there is an
exact sequence
\begin{equation}\label{PCseq}bo_*(P_1)\mapright{g_*} \pi_*(C\w P_1)\mapright{\delta_*} \pi_*(X\w\tmf\w\Sigma P_1)\mapright{\ft_*} bo_*(\Sigma P_1).\end{equation}
As is well-known, $bo_*(P_1)$ can be computed from $\ext_{A_1}(H^*(P_1))$, and from \ref{GMthm}(d), $\pi_*(C\w P_1)$ can be
computed from \begin{equation}\label{extPC}\ext_{A_2}(\Sigma^4 A_2/(\sq^4,\sq^{5,1})\otimes H^*(P_1)).\end{equation}
We can use Bruner's software to compute (\ref{extPC})
through a large range of dimensions, enough to see patterns. In order to prove that these patterns continue, $v_2^8$-periodicity,
which follows from the resolution in the proof of \ref{extCthm}, is very helpful, but we still need to prove what happens in
filtration less than 8 and dimension greater than 48. Most of our analysis will go into computing (\ref{extPC}), but we begin by analyzing (\ref{PCseq}).

It is convenient to use (\ref{PCseq}) to form a chart for $\pi_*(X\w\tmf\w \Sigma P_1)$ from
$$\phi\ext_{A_2}(\Sigma^4 A_2/(\sq^4,\sq^{5,1})\otimes H^*(P_1))\oplus \ext_{A_1}(H^*(\Sigma P_1)).$$
Recall that $\phi$ increases filtration by 1. The behavior for $10\le i\le 18$ is typical, and is depicted in
Diagram \ref{PCdiag}, in which black dots are from $\ext_{A_1}(H^*(\Sigma P_1))$ and $\circ$'s are from
$\phi\ext_{A_2}(\Sigma^4 A_2/(\sq^4,\sq^{5,1})\otimes H^*(P_1))$.

\medskip
\begin{minipage}{6.5in}
\begin{diag}\label{PCdiag}{Forming $\pi_*(X\w\tmf\w \Sigma P_1)$, $10\le*<18$}
\begin{center}
\begin{picture}(400,170)(80,0)
\def\mp{\multiput}
\def\elt{\circle*{3}}
\def\cir{\circle{3}}
\put(160,20){\line(0,1){120}}
\mp(160,20)(0,20){7}{\elt}
\put(120,100){\line(1,1){40}}
\mp(120,100)(20,20){2}{\elt}
\mp(157,60)(0,20){3}{\cir}
\mp(157,61)(0,20){2}{\line(0,1){18}}
\mp(119,60)(19,20){2}{\cir}
\mp(119,61)(19,20){2}{\line(1,1){18}}
\mp(119,40)(19,0){2}{\cir}
\put(160,20){\line(-1,1){20}}
\put(240,20){\line(0,1){140}}
\mp(240,20)(0,20){8}{\elt}
\mp(243,40)(0,20){5}{\cir}
\mp(243,41)(0,20){4}{\line(0,1){18}}
\mp(263,40)(20,20){2}{\cir}
\put(264,41){\line(1,1){18}}
\mp(246,60)(0,20){2}{\cir}
\put(246,61){\line(0,1){18}}
\mp(262,60)(20,20){2}{\cir}
\put(263,61){\line(1,1){18}}
\put(244,40){\line(1,1){19}}
\put(100,20){\line(1,0){200}}
\put(241,20){\line(1,1){20}}
\put(155,5){$12$}
\put(236,5){$16$}
\end{picture}
\end{center}
\end{diag}
\end{minipage}

The content in this chart is the $d_1$-differential from $(12,0)$ and the $\eta$-extension from $(16,0)$.
These are generalized and proved in Theorem \ref{XPCthm}.
\begin{thm} \label{XPCthm} In (\ref{PCseq}),
\begin{itemize}
\item $bo_{8i+3}(P_1)\mapright{g_*}\pi_{8i+3}(C\w P_1)$
is nontrivial.
\item There is an element $\gamma_{8i}\in \pi_{8i}(X\w\tmf\w\Sigma P_1)$ such that $\ft_*(\gamma_{8i})$
has Adams filtration 0, and $\eta \gamma_{8i}=\delta_*(y_{8i})\ne0$ with $y_{8i}$ of Adams filtration 0
in $\pi_{8i+1}(C\w P_1)$.
\end{itemize}
\end{thm}
\begin{proof} We will see in \ref{PCthm} that
$\ext_A^{0,8i+3}(H^*(C\w P_1))\approx\zt$ with nonzero class $\iota_4\otimes x_{8i-1}$. The morphism $g_*$
is induced by
$$\begin{matrix}
\Sigma^4A/(\sq^4,\sq^{5,1})\otimes H^*P_1&\to &A/\!/A_1\otimes H^*P_1&\approx& A\otimes_{A_1}H^*P_1\\
\iota_4\otimes x_{8i-1}&\mapsto&\sq^4\otimes x_{8i-1}&\leftrightarrow&\sq^4(1\otimes x_{8i-1})+1\otimes x_{8i+3},
\end{matrix}$$
which proves the first statement. The $\eta$-extension follows similarly from
$$\iota_4\otimes x_{8i-3}\mapsto \sq^4\otimes x_{8i-3}\leftrightarrow \sq^4(1\otimes x_{8i-3})+\sq^2(1\otimes
x_{8i-1}).$$
To know that the class $\gamma_{8i}$ is nonzero in $\pi_{8i}(-)$, we use \ref{PCthm} to see  that, unless $i\equiv 5$ mod 8, the only
possible target of a differential from $\gamma_{8i}$ is ruled out by $h_2$-naturality. If $i\equiv5$ mod 8,
the differential, if nonzero in the ASS of $\Sigma P_1$, would have to also be nonzero in the ASS of the cofiber $R$ of the
Kahn-Priddy map $\lambda:P_1\to S^0$, but it is ruled out there by $h_2$-naturality.
\end{proof}

Let $\Ct=A_2/(\sq^4,\sq^{5,1})$. A good way to obtain $\ext_{A_2}(\Ct\otimes H^*(P_1))$ begins by computing
$\ext_{A_2}(\Ct\otimes Q)$, where $Q$ is the $A_2$-module which has as its only nonzero classes $x_i$ for
$i\ge1$ and $i\in\{-9,-5,-3,-2,-1\}$ with $\sq^jx_i=\binom ijx_{i+j}$. Then $Q$ is an extension of copies of $\Sigma^{8i-1}A_2/\!/A_1$
for $i\ge-1$. See \cite[p.299]{DM}. Thus there is a spectral sequence converging to $\ext_{A_2}(\Ct\otimes Q)$ with
$$E_2^{*,*}=\bigoplus_{i\ge-1}\ext_{A_1}^{*,*}(\Sigma^{8i-1}\Ct).$$
One easily computes $\ext_{A_1}(\Ct)$ to be as in Diagram \ref{Ctdiag}, from which it is immediate
that the spectral sequence collapses and
\begin{equation}\label{Q}\ext_{A_2}(\Ct\otimes Q)\approx\bigoplus_{i\ge-1}\ext_{A_1}(\Sigma^{8i-1}\Ct).\end{equation}
We obtain that, in grading $8i+5$, $\ext_{A_2}(\Ct\otimes Q)$ has a $bo_*$ beginning in filtration $s$ for all nonnegative $s\le 4i+1$
except $s=4i$. This will explain the low-filtration form of Diagrams \ref{P1diag1} and \ref{P1diag2}.

\medskip
\begin{minipage}{6.5in}
\begin{diag}\label{Ctdiag}{$\ext_{A_1}(\Ct)$}
\begin{center}
\begin{picture}(400,185)(40,0)
\def\mp{\multiput}
\def\elt{\circle*{3}}
\put(0,20){\elt}
\put(140,20){\elt}
\mp(120,20)(20,20){2}{\elt}
\put(120,20){\line(1,1){20}}
\put(100,40){\vector(0,1){140}}
\mp(100,40)(0,20){5}{\elt}
\put(100,40){\line(1,1){40}}
\mp(120,60)(20,20){2}{\elt}
\put(180,40){\vector(0,1){140}}
\mp(180,40)(0,20){6}{\elt}
\put(183,60){\vector(0,1){120}}
\mp(183,60)(0,20){5}{\elt}
\put(186,100){\vector(0,1){80}}
\mp(186,100)(0,20){3}{\elt}
\put(262,20){\vector(0,1){160}}
\mp(262,20)(0,20){8}{\elt}
\put(262,20){\line(1,1){40}}
\mp(282,40)(20,20){2}{\elt}
\put(265,60){\vector(0,1){120}}
\mp(265,60)(0,20){6}{\elt}
\put(265,60){\line(1,1){38}}
\mp(265,60)(19,19){3}{\elt}
\put(268,80){\vector(0,1){100}}
\mp(268,80)(0,20){5}{\elt}
\put(268,80){\line(1,1){36}}
\mp(268,80)(18,18){3}{\elt}
\put(271,120){\vector(0,1){60}}
\mp(271,120)(0,20){3}{\elt}
\put(271,120){\line(1,1){34}}
\mp(271,120)(17,17){3}{\elt}
\mp(320,130)(8,8){3}{$\cdot$}
\put(-5,20){\line(1,0){320}}
\put(-2,5){$0$}
\put(97,5){$5$}
\put(180,5){$9$}
\put(261,5){$13$}
\end{picture}
\end{center}
\end{diag}
\end{minipage}

There is a short exact sequence of $A_2$-modules
$$0\to H^*P_1\to Q\to \Sigma^{-9}M_7\oplus\Sigma^{-1}\zt\to0,$$
and also after tensoring with $\Ct$. Thus there is an exact sequence
\begin{eqnarray}\label{QLES}&&\ext_{A_2}^s(\Ct\ot\Sigma^{-9}M_7)\oplus\ext_{A_2}^s(\Sigma^{-1}\Ct)\to\ext_{A_2}^s(\Ct\ot Q)\\
&\to&
\ext_{A_2}^s(\Ct\ot H^*P_1)\to\ext_{A_2}^{s+1}(\Ct\ot\Sigma^{-9}M_7)\oplus\ext_{A_2}^{s+1}(\Sigma^{-1}\Ct).\nonumber
\end{eqnarray}

In Theorem \ref{extCthm} and Diagram \ref{bigdiag}, we computed and displayed $\ext_{A_2}(\Ct)$.
A nice computation
of $\ext_{A_2}(\Ct\otimes M_7)$ can be obtained by tensoring the exact sequence at the beginning of the proof of \ref{extCthm}
with $M_7$. This yields a spectral sequence computing $\ext_{A_2}(\Ct\otimes M_7)$ from things such as
$\ext_{A_2}(M_7\otimes A_2)$, which is just four $\zt$'s, and $\ext_{A_2}(M_7\otimes A_2/\!/A_1)\approx\ext_{A_1}(M_7)$,
which is $bo_*\oplus\Sigma^4 bsp_*$. The resulting spectral sequence has only a very few possible differentials, which are
most easily settled using Bruner's software, although they can be settled without it. Both $\ext_{A_2}(\Ct)$ and $\ext_{A_2}(\Ct\ot M_7)$
have lower vanishing lines.
From these and the exact sequence, we obtain that
$$\ext_{A_2}^{s,t}(\Ct\ot Q)\to\ext_{A_2}^{s,t}(\Ct\ot H^*P_1)$$
is an isomorphism if $s\le8$ and $t-s\ge53$.

Thus a Bruner calculation of $\ext_{A_2}^{s,t}(\Ct\ot H^*P_1)$ for $t-s\le53$, which is easily done and is
consistent with Theorem \ref{PCthm}, together with the complete description of $\ext_{A_2}(\Ct\ot Q)$ in
(\ref{Q}) and \ref{Ctdiag} and $v_2^8$-periodicity, gives a complete determination of the groups
$\ext_{A_2}^{s,t}(\Ct\ot H^*P_1)$.
Note that the Bruner software is not absolutely necessary for this. First of all, it is just a finite calculation,
and secondly there are rather simple patterns for the boundary homomorphism in (\ref{QLES}), which could be determined
directly.

There is one more thing required in order to determine the chart for $\ext_{A_2}^{s,t}(\Ct\ot H^*P_1)$,
and the resulting $\pi_*(C\w P_1)$. In dimensions greater than 53 and congruent to 0 mod 4,
we know from the determination of $\ext_{A_2}(\Ct\ot Q)$
that in filtration $\le8$ $\ext_{A_2}^{s,t}(\Ct\ot H^*P_1)$ has $h_0$-towers beginning in each filtration ($>0$ in
dimension 0 mod 8),
and we know from the Bruner calculation and periodicity that in high filtration it has towers which end in every second filtration
coming down from a certain maximum filtration. But how do we know the way these match up? We must show that, as suggested
in Diagrams \ref{P1diag1} and \ref{P1diag2}, the lowest bottoms match up with the highest tops.

One way to do this is to use the spectral sequence which builds $\ext_{A_2}(\Ct\ot H^*P_1)$ from
\begin{equation}\label{phisum}\bigoplus_{s\ge0}\phi^s\ext_{A_2}^{*,*}(\Sigma^{-s}C_s\ot H^*P_1),\end{equation}
where $C_s$ are the $A_2$-modules in the resolution of $\Ct$ at the beginning of the proof of \ref{extCthm}.
The $s$-summand provides a bunch of $\zt$'s at height $s$ in the resulting chart (coming from $\phi^s\ext^0(-)$)
together with the portion of Diagram \ref{Pschart} consisting of towers beginning at height $s$.
Note that there are no such towers when $s=0$.

\medskip
\begin{minipage}{6.5in}
\begin{diag}\label{Pschart}{Portion of spectral sequence building $\ext_{A_2}(\Ct\ot H^*P_1)$}
\begin{center}
\begin{picture}(450,255)(30,0)
\def\mp{\multiput}
\def\elt{\circle*{3}}
\put(0,30){\line(1,1){30}}
\mp(0,30)(15,15){3}{\elt}
\put(30,30){\line(0,1){30}}
\mp(30,30)(0,15){2}{\elt}
\put(90,30){\line(0,1){45}}
\mp(90,30)(0,15){4}{\elt}
\put(87,45){\elt}
\mp(150,30)(0,15){7}{\elt}
\put(150,30){\line(0,1){90}}
\put(120,90){\line(1,1){30}}
\mp(120,90)(15,15){2}{\elt}
\put(147,45){\line(0,1){45}}
\mp(147,45)(0,15){4}{\elt}
\put(119,62){\line(1,1){28}}
\mp(119,62)(14,14){2}{\elt}
\put(212,30){\line(0,1){105}}
\mp(212,30)(0,15){8}{\elt}
\put(209,45){\line(0,1){60}}
\mp(209,45)(0,15){5}{\elt}
\put(206,60){\line(0,1){15}}
\mp(206,60)(0,15){2}{\elt}
\put(275,30){\line(0,1){150}}
\mp(275,30)(0,15){11}{\elt}
\put(245,150){\line(1,1){30}}
\mp(245,150)(15,15){2}{\elt}
\put(272,45){\line(0,1){105}}
\mp(272,45)(0,15){8}{\elt}
\put(242,122){\line(1,1){28}}
\mp(242,122)(14,14){2}{\elt}
\put(269,60){\line(0,1){60}}
\mp(269,60)(0,15){5}{\elt}
\put(269,120){\line(-1,-1){26}}
\mp(269,120)(-13,-13){3}{\elt}
\put(266,75){\line(0,1){15}}
\mp(266,75)(0,15){2}{\elt}
\put(266,90){\line(-1,-1){13}}
\mp(266,90)(-13,-13){2}{\elt}
\put(335,30){\line(0,1){165}}
\mp(335,30)(0,15){12}{\elt}
\put(332,45){\line(0,1){120}}
\mp(332,45)(0,15){9}{\elt}
\put(329,60){\line(0,1){75}}
\mp(329,60)(0,15){6}{\elt}
\put(326,75){\line(0,1){30}}
\mp(326,75)(0,15){3}{\elt}
\put(396,30){\line(0,1){210}}
\mp(396,30)(0,15){15}{\elt}
\put(396,240){\line(-1,-1){32}}
\mp(396,240)(-16,-16){3}{\elt}
\put(393,45){\line(0,1){165}}
\mp(393,45)(0,15){12}{\elt}
\put(393,210){\line(-1,-1){30}}
\mp(393,210)(-15,-15){3}{\elt}
\put(390,60){\line(0,1){120}}
\mp(390,60)(0,15){9}{\elt}
\put(390,180){\line(-1,-1){28}}
\mp(390,180)(-14,-14){3}{\elt}
\put(387,75){\line(0,1){75}}
\mp(387,75)(0,15){6}{\elt}
\put(387,150){\line(-1,-1){26}}
\mp(387,150)(-13,-13){3}{\elt}
\put(384,90){\line(0,1){30}}
\mp(384,90)(0,15){3}{\elt}
\put(384,120){\line(-1,-1){24}}
\mp(384,120)(-12,-12){3}{\elt}
\put(456,30){\line(0,1){225}}
\mp(456,30)(0,15){16}{\elt}
\put(453,45){\line(0,1){180}}
\mp(453,45)(0,15){13}{\elt}
\put(450,60){\line(0,1){135}}
\mp(450,60)(0,15){10}{\elt}
\put(447,75){\line(0,1){90}}
\mp(447,75)(0,15){7}{\elt}
\put(444,90){\line(0,1){45}}
\mp(444,90)(0,15){4}{\elt}
\put(441,90){\line(0,1){15}}
\mp(441,90)(0,15){2}{\elt}
\put(-5,15){\line(1,0){465}}
\put(-2,3){$6$}
\put(28,3){$8$}
\put(86,3){$12$}
\put(146,3){$16$}
\put(204,3){$20$}
\put(264,3){$24$}
\put(324,3){$28$}
\put(384,3){$32$}
\put(444,3){$36$}
\end{picture}
\end{center}
\end{diag}
\end{minipage}

\medskip
The desired form for the bottoms of the towers, as obtained from the complete description of $\ext_{A_2}(\Ct\ot Q)$ in
(\ref{Q}) and \ref{Ctdiag}, differs slightly from this, in that in dimensions congruent to 4 mod 8 most of the towers
should begin one filtration lower. This can only be accounted for by an extension from a $\zt$ from the next smaller $s$-value.

For example, in dimension 28, Diagram \ref{Pschart} shows towers beginning at height 1, 2, 3, and 4, coming from
summands $s=1$, 2, 3, and 4 in (\ref{phisum}) with tops at height 12, 10, 8, and 6, respectively. These
correspond to $\pi_{32}(C\w P_1)$, which, according to Theorem \ref{PCthm}, corresponds to $\pi_{32}(X\w\tmf\w \Sigma P_1)$ in
Diagram \ref{P1diag2} with its largest tower removed and filtrations decreased by 1; hence, towers beginning at height
0, 1, 2, and 3 ending at height 12, 10, 8, and 6. Then, for example, the tower in $\pi_{32}(C\w P_1)$ (corresponding to
$\ext_{A_2}^{*,*+28}(\Ct\ot H^*P_1)$) going from filtration 0 to 12 can only come, in the spectral sequence of
(\ref{phisum}), from the $s=1$ tower with an extension from a $\zt$ from $s=0$.

The main thing that was obtained from using $Q$ which was not easily obtained from (\ref{phisum}) is the
$\eta^2$-hooks on the bottom of towers. In (\ref{phisum}) these come about from the filtration-0 $\zt$'s
in the various $s$-summands in a complicated way, but they are clear in Diagram \ref{Ctdiag}.
The above remarks imply the following result, the computation of (\ref{extPC}), since there are no possible differentials
in the ASS.
\begin{thm}\label{PCthm} The ASS converging to $\pi_*(C\w P_1)$ has $E_2^{s,t}=\ext_{A_2}^{s,t}(\Sigma^4\Ct\ot H^*(P_1))$
and collapses. The description of $\pi_*(C\w P_1)$ can be obtained from that of $\pi_*(X\w\tmf \w\Sigma P_1)$ in Theorem \ref{P1thm}
by making the following changes:
\begin{itemize}
\item Remove summands in (\ref{cas}) for which $e_2=0$, (but do not remove $\eta a^{e_1}$ and $\eta^2 a^{e_1}$ when $e_1\equiv0$ mod 4);
\item Remove $b_{e_1,0}$ and $\eta b_{e_1,0}$ with $e_1\equiv2$ mod 4;
\item Add elements $c_{8i+3,0}$ of order 2 for $i\ge1$;
\item Decrease filtrations by 1.
\end{itemize}
\end{thm}

 The proof of Theorem \ref{P1thm} is now immediate from the exact sequence (\ref{PCseq}), Theorem \ref{PCthm}, and Theorem
 \ref{XPCthm}, which describes the only possible differentials and extensions in (\ref{PCseq}).

\begin{rmk} {\rm The way that we have chosen to describe these things is reversed from the way they are derived.
We first compute the groups in \ref{PCthm} and then use them to determine the groups in \ref{P1thm}. However, we are mostly
interested in \ref{P1thm}, and so we felt that it should be stated up front. It seemed like overkill to state the whole
thing again for $\pi_*(C\w P_1)$, since it is so similar.}\end{rmk}

\def\line{\rule{.6in}{.6pt}}

\end{document}